\documentclass{article}
\usepackage[margin=1in]{geometry}
\usepackage{graphicx}
\usepackage{amsmath}
\usepackage{amsfonts}
\usepackage{amssymb}
\usepackage{amsthm}
\usepackage[nocompress]{cite}
\usepackage{xcolor}
\usepackage{algorithm}
\usepackage{algorithmic}
\usepackage{footnote}
\usepackage{multirow}
\usepackage{hyperref}
\usepackage{overpic}
\usepackage[capitalise]{cleveref}
\usepackage{bm}
\usepackage{bbm}
\usepackage{stmaryrd}
\usepackage{framed}

\usepackage{tikz}
\usetikzlibrary{calc,external,patterns,shapes}

\crefformat{equation}{(#2#1#3)}
\crefrangeformat{equation}{(#3#1#4) to~(#5#2#6)}
\crefmultiformat{equation}{(#2#1#3)}{ and~(#2#1#3)}{, (#2#1#3)}{ and~(#2#1#3)}
\newtheorem{theorem}{Theorem}[section]
\newtheorem{remark}[theorem]{Remark}
\newtheorem{lemma}[theorem]{Lemma}
\newtheorem{proposition}[theorem]{Proposition}

\crefname{remark}{Remark}{Remarks}
\crefname{lemma}{Lemma}{Lemmas}
\crefname{proposition}{Proposition}{Propositions}
\crefname{corollary}{Corollary}{Corollaries}

\newcommand{\abs}[1]{\left|#1\right|}
\newcommand{\norm}[1]{\left\|#1\right\|}
\newcommand{\jump}[1]{\left\llbracket#1\right\rrbracket}

\newcommand{\bR}{\mathbb{R}}

\newcommand{\vf}{\mathbf{g}_N}
\newcommand{\vfcorr}[1]{\widetilde{\mathbf{g}}_{N,#1}}
\newcommand{\vforce}{\mathbf{f}}
\newcommand{\vg}{\mathbf{g}_D}
\newcommand{\vgcorr}[1]{\widetilde{\mathbf{g}}_{D,#1}}
\newcommand{\vn}{\mathbf{n}}
\newcommand{\vx}{\mathbf{x}}
\newcommand{\vy}{\mathbf{y}}
\newcommand{\vu}{\mathbf{u}}
\newcommand{\vv}{\mathbf{v}}

\newcommand{\vphi}{\bm\phi}

\newcommand{\vzero}{\mathbf{0}}

\newcommand{\uE}{{\bm\varepsilon}}
\newcommand{\uI}{\textup{I}}

\newcommand{\uT}{{\bm\sigma}}
\newcommand{\mcK}{\mathcal{K}}
\newcommand{\mcV}{\mathcal{V}}

\newcommand{\dw}{\partial \Omega}
\newcommand{\prodscal}[2]{\left(#1,#2\right)}

\usepackage{fancyhdr}
\title{Accelerating droplet-laden Stokes flow simulations with hierarchical surrogate modeling}
\newcommand{\email}[1]{\protect\href{mailto:#1}{#1}}
\author{Davide Pradovera\thanks{Stockholm University, Stockholm, Sweden (\email{davide.pradovera@math.su.se}).}
\and Thomas Frachon\thanks{KTH Royal Institute of Technology, Stockholm, Sweden (\email{frachon@kth.se}, \email{sara.zahedi@math.kth.se}).}
\and Sara Zahedi\footnotemark[2]}

\begin{document}

\maketitle

\begin{abstract}
    We present a surrogate modeling strategy for Stokes flows with liquid droplets suspended in a carrier fluid. Our approach is based on a multi-fidelity framework. At the lowest fidelity, droplets are treated as passive tracers, neglecting their influence on the ambient flow field. Building on this approximation, we derive a PDE that represents the current modeling error. This error equation is then solved approximately to correct the flow field and the procedure is iterated. Two fidelities are employed in an alternating fashion: Stokes flow in the absence of droplets and flow around a single droplet in free space. By systematically combining these models, the method captures droplet-flow, droplet-boundary, and droplet-droplet interactions. 
    In this work, the framework is developed and validated for circular, non-deforming droplets in two spatial dimensions. The geometric self-similarity of the droplets allows us to construct an efficient offline-online strategy based on the reuse of precomputed single-droplet solutions. Extensions to deformable droplets are also discussed.
    Numerical experiments demonstrate the accuracy and efficiency of the proposed surrogate in a variety of tests, including scenarios with up to $10^4$ droplets. Notably, we show that the proposed surrogate achieves substantially reduced computational cost compared to fully resolved multi-fluid simulations with state-of-the-art software.
\end{abstract}

\section{Introduction}
Multiphase flow systems arise in a wide range of applications, including microfluidics, porous media, ink-jet printing and biomedical applications. Accurate predictions for such systems increasingly rely on high-fidelity computer simulations\cite{Sethian, WOR12, Porous24, Revnummeth25}. In this work, we study incompressible multiphase Stokes flow involving a large number of immiscible droplets suspended in a carrier fluid.

From a computational point of view, a key difficulty arises when droplets are many and become highly concentrated or approach boundaries. Standard discretization techniques can only capture the dynamics in such regimes by resolving the narrow interfacial gaps and near-boundary interactions, making the resulting numerical schemes computationally demanding.

For example, standard finite-element methods rely on fine meshes that conform to the interfaces and therefore require repeated remeshing as the interfaces evolve. Unfitted methods such as CutFEM avoid remeshing but still require fine local resolution to accurately capture narrow gaps \cite{GrosReu2011, acta25}. Boundary integral methods face similar challenges, requiring specialized quadrature schemes and fine resolution around interfaces to accurately evaluate near-singular integrals arising from closely interacting droplets and boundaries \cite{OjaTor15}. 

Alternative particle-based methods, such as smoothed particle hydrodynamics \cite{Zhang2025} and the moving particle semi-implicit method \cite{Duan2017,Wu2023}, eliminate the need for mesh generation and remeshing. However, accurately resolving near-contact interactions still requires sufficiently fine particle resolutions. Moreover, these methods typically achieve lower accuracy than high-order finite-element or boundary-integral methods and often require careful tuning of numerical parameters (such as kernel functions and stabilization terms) to ensure numerical stability.

Point-particle models offer a further computationally efficient alternative, in which each particle (or droplet) is represented through singular (or regularized) force distributions and interactions are approximated using Green's-function-based expansions or related techniques. Such approaches have been successfully applied to low-Reynolds-number flows, including particle motion near boundaries \cite{SimMoMor18}. However, they are most naturally suited to regimes in which the particle size is small compared with both the characteristic flow scales and the particle-particle or particle-boundary separations.

The target problem in this work is instead a multiphase Stokes interface problem, where droplets are modeled as finite-size subdomains with viscosity jumps and interfacial stress-balance conditions. Our objective is therefore to accelerate particle-resolved simulations in regimes where finite-size effects, viscosity contrast, and close droplet-droplet or droplet-boundary interactions are important, rather than to replace point-particle models in their natural regime of applicability.

An additional challenge arises as the interfaces evolve in time, since the flow problem must be approximated and solved on a sequence of changing geometries.  In scenarios with many droplets, the fine spatial resolution required to capture near-contact interactions usually leads to expensive simulations. In this work, we present a surrogate model that enables efficient prediction of droplet evolution, even in densely packed, time-varying configurations. Conceptually, our surrogate combines ideas from both finite-element and boundary-integral approaches within an iterative correction framework. Boundary corrections are obtained by repeatedly solving a Stokes problem posed on the physical domain, using a finite-element discretization, while droplet corrections are represented through localized free-space single-droplet problems, which naturally admit boundary-integral formulations. The resulting iterative strategy retains the ability of finite-element discretizations to handle complex outer-boundary geometries while avoiding the need to explicitly resolve every droplet interface and interfacial gap.

For the time-dependent setting, we consider non-contacting configurations, allowing small positive interfacial gaps. The present  surrogate evolution strategy nevertheless restricts the droplets to rigid translations and does not account for shape deformation.

\subsection{Target problem}

We consider an incompressible multiphase Stokes system in a bounded Lipschitz domain $\Omega\subset\bR^d$, $d\in\{2,3\}$. The system consists of $M\geq1$ droplets of immiscible fluid in a surrounding carrier fluid. The droplets occupy the region\footnote{We use the subscript ``$B$'' (standing for ``bubbles'') rather than ``$D$'' (for ``droplets'') to avoid misunderstandings with Dirichlet boundary notation.} $\overline{\Omega_B}\subset\Omega$, which consists of $M$ connected components (the ``droplets''). Each droplet is bounded by a smooth interface, and the surrounding fluid occupies $\Omega_F=\Omega\setminus\overline{\Omega_B}$. The union of all droplet interfaces is denoted by $\Gamma=\partial\Omega_B$. To represent the different fluid properties, we introduce a piecewise constant viscosity field $\mu:\Omega\to\bR_{>0}$, with jumps across the droplet interfaces $\Gamma$.

Dirichlet and Neumann boundary conditions are prescribed on $\Gamma_D\subset\partial\Omega$ and $\Gamma_N=\partial\Omega\setminus\Gamma_D$, respectively, although other boundary conditions may also be considered. A schematic illustration of the computational domain and its boundaries is shown in \cref{fig:domain}. Across the interfaces $\Gamma$, the velocity is continuous, ensuring conservation of mass, while the jump in the normal stress balances the surface tension force.

\begin{figure}
    \centering
    \vspace{-3mm}
    \includegraphics{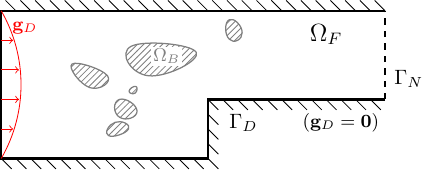}
    \vspace{-2mm}
    \caption{Example: domain of the Stokes interface problem.}
    \label{fig:domain}
\end{figure}

The interfaces evolve with the fluid velocity. At each time $t \in [0,T]$, the flow satisfies the Stokes equations:
\begin{equation}\label{eq:cutfemtime}
    \begin{cases}
        \dot\vx(t)=\vu(t) & \text{on } \Gamma(t),\\
        -\nabla\cdot\uT(\vu(t),p(t);\mu)=\vforce & \text{in } \Omega\setminus\Gamma(t),\\
        \nabla\cdot\vu(t)=0 & \text{in } \Omega\setminus\Gamma(t),\\
        \jump{\vu(t)}=\vzero & \text{on } \Gamma(t),\\
        \jump{\uT(\vu(t),p(t);\mu)\vn}=\gamma\kappa\vn & \text{on } \Gamma(t),\\
        \vu(t)=\vg & \text{on } \Gamma_D,\\
        \uT(\vu(t),p(t);\mu)\vn=\vf & \text{on } \Gamma_N,
    \end{cases}
\end{equation}
with $\vx(t)$ denoting the position of the interface, whereas $\vu$ and $p$ denote the velocity and pressure fields, respectively. Furthermore, $\uT(\vu,p;\mu)=2\mu\uE(\vu)-p\uI$ is the viscous stress tensor, $\gamma$ is the (piecewise constant) surface tension coefficient, $\kappa:\Gamma\to\bR$ is the mean curvature of the interface $\Gamma$ and $\vn$ is the unit normal to the interface (pointing towards $\Omega_F$) and on the outer boundary (pointing outwards). We use double square brackets to denote jumps across $\Gamma$, i.e., $\jump{\vu}:=\vu|_{\Omega_F}-\vu|_{\Omega_B}$. 

The volume forcing term is $\vforce\in[H^{-1}(\Omega)]^d$, whereas the Dirichlet and Neumann boundary data are given by $\vg\in[H^{1/2}(\Gamma_D)]^d$ and $\vf\in[H^{-1/2}(\Gamma_N)]^d$, respectively\footnote{If $\Gamma_N=\emptyset$, well-posedness requires supplementing the problem with a zero-mean condition on the pressure, $\int_\Omega p=0$. This condition can be imposed in post-processing once a solution has been obtained  and does not affect the present discussion.}. For simplicity, we assume that the boundary data are time-independent, so that the only time dependence in \cref{eq:cutfemtime} arises through the  evolution of the interface $\Gamma(t)$. This assumption can be relaxed, as the main challenge lies in the treatment of the evolving interface $\Gamma(t)$. Assuming that $\Gamma(t)$ is sufficiently smooth, there exists a weak solution $(\vu(t),p(t))\in [H^1(\Omega)]^d\times L^2(\Omega)$ for each fixed $t\in[0,T]$. 

The linearity (in $\vu$ and $p$) of the multiphase Stokes problem \cref{eq:cutfemtime} makes it particularly attractive for the development of compact surrogate models. While the more general multiphase Navier-Stokes equations are required when fluid inertia cannot be neglected, their nonlinear nature introduces substantial additional complexity. At the same time, the Stokes setting accurately describes a broad class of low-Reynolds-number multiphase flows and is already sufficiently rich to capture the geometric and interfacial challenges that motivate this work.

While the formulation \cref{eq:cutfemtime} allows for deforming interfaces, the present work focuses on the simplified setting of circular, non-deforming droplets. This setting is already sufficiently rich to capture viscosity-contrast effects, droplet--droplet interactions, and droplet--boundary interactions, while avoiding the additional complications associated with deforming interfaces. Moreover, as discussed in \cref{sec:basis}, we develop an efficient offline-online acceleration strategy that exploits the geometric self-similarity of circular droplets through the reuse of precomputed single-droplet solutions. For this reason, all numerical examples are performed in two spatial dimensions with circular, non-deforming droplets. Extensions to deformable droplets are discussed in \cref{sec:extensions}.

\subsection{Surrogate modeling}
There is growing interest in surrogate models that provide accurate and efficient predictions of multiphase flows, particularly in regimes where fully resolved simulations are computationally challenging. Machine learning (ML) algorithms are increasingly employed to construct such surrogates, with the aim of replacing expensive high-fidelity simulations with computationally efficient predictions. ML-based approaches have been used, for example, to predict droplet breakup \cite{Cundy24}, flow fields around obstacles such as spinning spheres \cite{Liu24}, pairwise hydrodynamic interactions in Stokes flows \cite{Elmestikawy24}, and the spatiotemporal evolution of microstructures \cite{Yang21}. Data-driven reduced-order models have also been used to predict droplet dynamics directly from observations. For instance, Zhuang et al. \cite{Zhuang2022} combine a reduced-order latent-space representation with recurrent neural networks and ensemble data assimilation to predict droplet interactions in a microfluidic device. Most closely related to the computational problem considered here is the machine-learning-augmented reduced model proposed in \cite{Biros19}, which also studies systems of many particles interacting through Stokes flow, although in the setting of rigid particles. Beyond purely data-driven surrogates, multiscale offline-online strategies have also been proposed. For example, Gim{\'e}nez et al.~\cite{Gimenez2021} combine offline representative-volume-element simulations of droplet-turbulence interactions with an online coarse-scale solver.

Along similar lines, machine learning has been used to model hydrodynamic forces acting on rigid particles, which can then be incorporated into coarse-scale flow solvers to improve the accuracy of particle-dynamics simulations \cite{minakowska_finite_2021,von_wahl_using_2021}. For the stationary Stokes problem with rigid particles, we also mention the recent work \cite{BroBarTor25}, which combines the method of fundamental solutions with the method of images to obtain a computationally cheaper alternative to standard boundary integral formulations and demonstrates simulations with up to 2000 fixed spheres.

In regimes with many droplets and small interfacial separations, existing high-fidelity solvers require extreme mesh refinement near narrow gaps, while existing ML surrogates often lack the physical transparency needed to guarantee reliability and accuracy, particularly outside the regimes represented by the training data. To address these challenges, this work introduces a novel hierarchical, physics-informed surrogate modeling framework for the multiphase Stokes problem \cref{eq:cutfemtime}. The core ideas and contributions of our approach are:
\begin{itemize}
    \item \emph{Physics-based error correction.} Instead of treating the surrogate as a black box, we propose a multi-fidelity iterative strategy. Beginning with a baseline Stokes solution that neglects interfaces, we systematically derive and solve a sequence of PDE-based error equations to incrementally recover the high-fidelity solution.
    \item \emph{Replacing refinement with iteration.} To avoid the high complexity of solving global error equations in tightly packed domains, we introduce an incremental approximation strategy. Corrections are constructed locally through a superposition of the field contributions of individual droplets, reducing both memory requirements and computational cost.
    \item \emph{Accuracy at low cost.} Numerical experiments demonstrate that the proposed framework achieves errors below $\sim 0.1\%$, compared to fully resolved high-fidelity simulations, at substantially lower computational cost.
\end{itemize}
The remainder of the paper is organized as follows. \cref{sec:preliminaries} introduces the multi-fidelity framework and the associated error equations. \cref{sec:surrogate} presents the surrogate-modeling strategy and the offline-online acceleration. Convergence and complexity are discussed in \cref{sec:termination,sec:timing}, respectively. Possible extensions to deformable droplets are outlined in \cref{sec:extensions}.  Numerical experiments are presented in \cref{sec:numerics}, and concluding remarks are given in \cref{sec:conc}.

\section{Preliminaries}\label{sec:preliminaries}
Since the only time-dependence in \cref{eq:cutfemtime} is due to the evolving interface $\Gamma(t)$, we first consider the stationary Stokes problem by itself, with $\Gamma=\Gamma(\overline t)$ for some $\overline t$:
\begin{equation}\label{eq:cutfem}
    \begin{cases}
        -\nabla\cdot\uT(\vu,p;\mu)=\vforce & \text{in } \Omega\setminus\Gamma=\Omega_B \cup \Omega_F,\\
        \nabla\cdot\vu=0 & \text{in } \Omega_B \cup \Omega_F,\\
        \jump{\vu}=\vzero & \text{on } \Gamma,\\
        \jump{\uT(\vu,p;\mu)\vn}=\gamma\kappa\vn & \text{on } \Gamma,\\
        \vu=\vg & \text{on } \Gamma_D,\\
        \uT(\vu,p;\mu)\vn=\vf & \text{on } \Gamma_N.
    \end{cases}
\end{equation}
This problem naturally arises whenever discretizing \cref{eq:cutfemtime} with respect to time. As such, we set the solution $(\vu,p)$ of \cref{eq:cutfem} as our preliminary target for surrogate modeling. We get back to the time-dependent problem in \cref{sec:time}.

We further assume that the surrounding fluid $\Omega_F$ has homogeneous viscosity $\mu_F\in\mathbb R_{>0}$, i.e., $\mu|_{\Omega_F}\equiv\mu_F$. Moreover, $\mu$ is assumed to be constant on each droplet, although different droplets may have different viscosities.

\subsection{Lowest fidelity}\label{sec:preliminaries:low}
As lowest fidelity, we consider a Stokes problem in the absence of droplets. Specifically, we consider the case where the entire domain $\Omega$ is filled with a single fluid of homogeneous ``nominal'' viscosity $\mu_F$. This corresponds to modeling the droplets as \emph{passive tracers}, neglecting both viscosity contrast and surface tension. The governing equations then read
\begin{equation}\label{eq:fem}
    \begin{cases}
        -\nabla\cdot\uT(\vu_0,p_0;\mu_F)=\vforce & \text{in } \Omega,\\
        \nabla\cdot\vu_0=0 & \text{in } \Omega,\\
        \vu_0=\vg & \text{on } \Gamma_D,\\
        \uT(\vu_0,p_0;\mu_F)\vn=\vf & \text{on } \Gamma_N.
    \end{cases}
\end{equation}
Note, in particular, that the boundary conditions are consistent with those of \cref{eq:cutfem}, given our assumption on $\mu$. Standard arguments \cite{GirRav86} show that this problem is well-posed, with solution $(\vu_0,p_0)\in[H^1(\Omega)]^d\times L^2(\Omega)$. In addition, this solution satisfies the following interface conditions across the ``immersed interfaces'' $\Gamma\subset\Omega$:
\begin{equation}\label{eq:fematgamma}
    \jump{\vu_0}=\vzero=\jump{\uT(\vu_0,p_0;\mu_F)\vn} \text{ on } \Gamma.
\end{equation}
The first identity follows from the fact that $\vu_0\in [H^1(\Omega)]^d$, which admits a unique trace on $\Gamma$. Moreover, since $(\vu_0,p_0)\in [H^1(\Omega)]^d \times L^2(\Omega)$ satisfies $-\nabla\cdot \uT(\vu_0,p_0;\mu_F)=\vzero$ in the sense of distributions, an application of the divergence theorem on the subdomains separated by $\Gamma$ shows that no singular interface term arises, and hence $\jump{\uT(\vu_0,p_0;\mu_F)\vn}=\vzero$; see, e.g.,\cite{GirRav86} or standard derivations of interface jump conditions in Stokes/elasticity problems.

We can characterize the lowest-fidelity error $(\vu-\vu_0,p-p_0)$ as follows.

\begin{proposition}\label{prop:correctbase}
    Let $(\vu,p)$ and $(\vu_0,p_0)$ solve \cref{eq:cutfem} and \cref{eq:fem}, respectively. The error fields $\vv^{(0)}:=\vu-\vu_0$ and $q^{(0)}:=p-p_0$ satisfy
    \begin{equation}\label{eq:butterfly}
        \begin{cases}
            -\nabla\cdot\uT(\vv^{(0)},q^{(0)};\mu)=2(\mu-\mu_F)\nabla\cdot\uE(\vu_0)&\textup{in }\Omega_B\cup\Omega_F,\\
            \nabla\cdot\vv^{(0)}=0&\textup{in }\Omega_B\cup\Omega_F,\\
            \jump{\vv^{(0)}}=\vzero&\textup{on }\Gamma,\\
            \jump{\uT(\vv^{(0)},q^{(0)};\mu)\vn}=\left(\gamma\kappa\uI-2\jump{\mu}\uE(\vu_0)|_{\Omega_B}\right)\vn&\textup{on }\Gamma,\\
            \vv^{(0)}=\vzero&\textup{on }\Gamma_D,\\
            \uT(\vv^{(0)},q^{(0)};\mu)\vn=\vzero&\textup{on }\Gamma_N.
        \end{cases}
    \end{equation}
\end{proposition}

\begin{proof}
    The result follows by subtracting the relations in \cref{eq:fem} and \cref{eq:fematgamma} from the corresponding relations in \cref{eq:cutfem}, using the bilinearity of $\uT$. This is also a special case of \cref{prop:correctfull} below.
\end{proof}

\begin{figure}[tb!]
    \centering
    \includegraphics{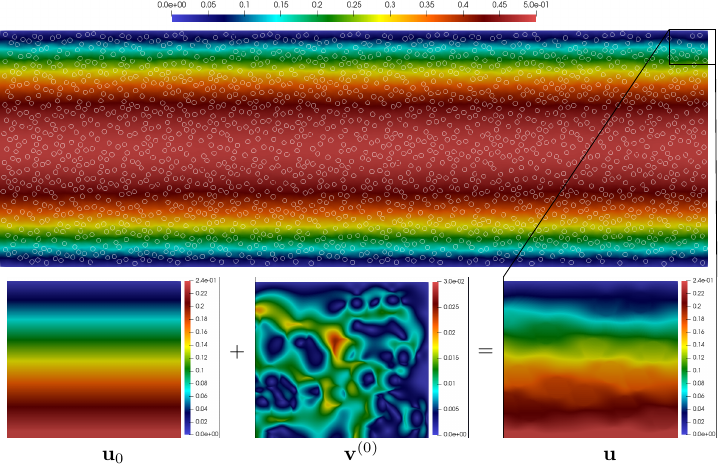}\vspace{-1mm}
    \caption{Example of Stokes flow in a channel with $M=2000$ circular droplets, represented with white outlines (full details are provided in \cref{sec:num:scaling}). Top: magnitude of velocity field. Bottom row: the solution can be decomposed into  a baseline, no-droplet flow field (Poiseuille flow, lowest fidelity) and error field, cf.~\cref{prop:correctbase}.}
    \label{fig:scaling_qual}
\end{figure}

The decomposition of the exact field into a lowest-fidelity field and an error field is illustrated in \cref{fig:scaling_qual}. Identifying an equation governing the lowest-fidelity error is crucial to our approach. Indeed, to improve the approximation quality, a second fidelity level is introduced on top of $(\vu_0,p_0)$, based on \cref{eq:butterfly}. More specifically, a two-fidelity approximation of $(\vu,p)$ will be constructed as $(\vu_0+\vv_0,p_0+q_0)$, where $(\vv_0,q_0)$ approximates the solution $(\vv^{(0)},q^{(0)})$ of \cref{eq:butterfly}.

Note that we seek a surrogate $(\vv_0,q_0)$ rather than directly computing $(\vv^{(0)},q^{(0)})$, since solving \cref{eq:butterfly} incurs essentially the same computational cost as solving the original problem \cref{eq:cutfem}. In particular, both problems are affected by the same computational challenges associated with resolving the droplet interfaces.

\subsection{Further fidelities}\label{sec:singledroplet}
As outlined above, in order to recover computational tractability of the error equation \cref{eq:butterfly}, we consider a second auxiliary problem. Only the volume and interface effects of \emph{a single droplet} are kept, while all other droplets, as well as the domain boundaries, are removed: given a connected $\Omega_B^i\subset\Omega_B$ (a single droplet), we look at
\begin{equation}\label{eq:singledroplet}
    \begin{cases}
        -\nabla\cdot\uT(\vv,q;\mu)=2(\mu-\mu_F)\nabla\cdot\uE(\vphi) & \text{in } \Omega_B^i,\\
        -\nabla\cdot\uT(\vv,q;\mu_F)=\vzero & \text{in } \bR^d\setminus\overline{\Omega_B^i},\\
        \nabla\cdot\vv=0 & \text{in } \bR^d\setminus\partial\Omega_B^i,\\
        \jump{\vv}=\vzero & \text{on } \partial\Omega_B^i,\\
        \jump{\uT(\vv,q;\mu)\vn}=\left(\widetilde\gamma\kappa\uI-2\jump{\mu}\uE(\vphi)\right)\vn & \text{on } \partial\Omega_B^i,\\
        \lim_{\norm{\vx}\to\infty}\vv(\vx)=\vzero.&
    \end{cases}
\end{equation}
The surface tension coefficient $\widetilde\gamma$ and the generic flow field $\vphi\in[H^1(\Omega_B^i)]^d$ are meant as placeholders, although initially equal to $\gamma$ and $\vu_0$, respectively, cf.~\cref{prop:correctbase}.

This formulation may be viewed as a localized version of the consistency equations in~\cref{prop:correctbase}, under the assumption that only the $i$-th droplet is present. More specifically, the viscosity field is replaced by
\begin{equation}\label{eq:mupartial}
    \mu_F+(\mu|_{\Omega_B^i}-\mu_F)\chi_{\Omega_B^i}.
\end{equation}
Here and elsewhere, $\chi_A$ denotes the indicator function of $A$, equal to $1$ on $A$ and $0$ elsewhere.

Note that \cref{eq:singledroplet} is cast in free $\bR^d$-space and a zero velocity far-field is imposed for closure. This achieves four beneficial effects:
\begin{itemize}
    \item When, in the original problem \cref{eq:cutfem}, the gap between a droplet $\Omega_B^i$ and the domain boundary $\partial\Omega$ is small, one is required to employ highly refined discretizations. See \cref{rem:traditional} for more details. This issue is completely absent in \cref{eq:singledroplet} due to the boundary removal.
    \item If two droplets $\Omega_B^i$ and $\Omega_B^{i'}$ are (geometrically) similar to each other, the corresponding problems \cref{eq:singledroplet} can be mapped one to the other through simple affine mappings $\Omega_B^i\mapsto\Omega_B^{i'}$ and changes in $\widetilde\gamma$ and $\vphi$. As discussed in \cref{sec:basis}, this can be exploited for efficiency gains.
    \item Non-homogeneous forcing terms are present only at $\overline{\Omega_B^i}$ and decay conditions are imposed at infinity, so that the magnitude of the solution $(\vv,q)$ can be expected to decrease as the distance from the droplet increases. See also \cref{sec:termination} below. In a fundamental way, this achieves \emph{localization} in each droplet's effect on the flow, which is crucial for the convergence of our approach.
    \item Problem \cref{eq:singledroplet} naturally lends itself to a boundary-integral-based discretization, where the infinite domain size is actually beneficial rather than detrimental \cite{tornberg_fast_2008}. In particular, approximating the solution to this problem through boundary-integral equations is much simpler than if boundary conditions at $\partial\Omega$ had been imposed.
\end{itemize}

\section{Surrogate modeling}\label{sec:surrogate}
In what follows, we restrict our attention to configurations where all $M$ droplets are \emph{circular}:
\begin{equation*}
    \Omega_B=\bigcup_{i=1}^M\underbrace{B(\vx_i;r_i)}_{\Omega_B^i},
\end{equation*}
where centers $\vx_i$ and radii $r_i$ are suitably chosen so that droplets $\overline{\Omega_B^i}$ are pairwise disjoint and $\overline{\Omega_B}\subset\Omega$, in line with our working assumptions. This condition greatly simplifies both the analytical derivations and the construction of surrogate models, while still capturing the essential features of the problem. In \cref{sec:extensions} we discuss possible extensions to non-circular droplets, which might be of particular relevance when simulating time-dependent flows, where interfaces naturally deform.

If all droplets are circular, then the effects of surface tension may be accounted for in an exact way through pressure jumps. Specifically, recalling that surface tension is uniform on each droplet's boundary but possibly different across droplets, consider
\begin{equation*}
    \widetilde p=\sum_{i=1}^M\frac{\gamma_i}{r_i}\chi_{\Omega_B^i},
\end{equation*}
where $\gamma_i:=\gamma|_{\partial \Omega_B^i}\in\bR$. Then $(\vu,p-\widetilde p)$ satisfies
\begin{equation*}
    \begin{cases}
        -\nabla\cdot\uT(\vu,p-\widetilde p;\mu)=\vforce & \text{in } \Omega_B \cup \Omega_F,\\
        \nabla\cdot\vu=0 & \text{in } \Omega_B \cup \Omega_F,\\
        \jump{\vu}=\vzero & \text{on } \Gamma,\\
        \jump{\uT(\vu,p-\widetilde p;\mu)\vn}=\vzero & \text{on } \Gamma,\\
        \vu=\vg & \text{on } \Gamma_D,\\
        \uT(\vu,p-\widetilde p;\mu)\vn=\vf & \text{on } \Gamma_N,
    \end{cases}
\end{equation*}
cf.~\cref{eq:cutfem}. This is easily seen by linearity of the equation and by inspecting the curvature $\kappa\equiv r_i^{-1}$ on $\partial \Omega_B^i$.

For this reason, in describing our surrogate-modeling strategy, we may assume that no surface tension is present without loss of generality: $\gamma=0$. Indeed, any nonzero amount of surface tension may be accounted for by just adding $\widetilde p$ to the surrogate in a post-processing step.

\subsection{Initial correction of droplet effects}\label{sec:full:first}

As outlined in \cref{sec:preliminaries}, we take the solution $(\vu_0,p_0)$ of the ``no-droplet'' Stokes problem \cref{eq:fem} with the prescribed boundary data $(\vg,\vf)$ as a first low-fidelity approximation. This solution satisfies all boundary conditions on $\partial\Omega$ but fails to satisfy the stress balance equation within the droplets $\Omega_B$ and the interfacial force balance at $\Gamma=\partial\Omega_B$.

This situation is captured in \cref{prop:correctbase}, which shows that, to steer $(\vu_0,p_0)$ toward the true solution $(\vu,p)$ of \cref{eq:cutfem}, one would in principle need to add a globally supported correction $(\vv^{(0)},q^{(0)})$ solving \cref{eq:butterfly}. However, computing such correction by solving \cref{eq:butterfly} is computationally unfeasible, since we would require dealing with all droplets' effects simultaneously. Instead, we construct a computationally cheaper surrogate correction $(\vv_0,q_0)\approx(\vv^{(0)},q^{(0)})$ via localized corrections.

First, note that the correction equation \cref{eq:butterfly} has non-homogeneous forcing terms supported over $\overline{\Omega_B}$, which, if $M$ non-overlapping droplets are present, consists of $M$ disjoint connected components. The trick to recover tractability is to replace the single problem \cref{eq:butterfly} with forcing terms at $M$ droplets with $M$ independent problems, each having forcing terms at a single droplet. More concretely, we consider a collection of $M$ problems resembling \cref{eq:singledroplet}, each targeting a different droplet $\Omega_B^i$, namely,
\begin{equation}\label{eq:singledropfull}
    \begin{cases}
        -\nabla\cdot\uT(\vv_{0,i},q_{0,i};\mu)=2(\mu-\mu_F)\nabla\cdot\uE(\vu_0)&\text{in }\Omega_B^i,\\
        -\nabla\cdot\uT(\vv_{0,i},q_{0,i};\mu_F)=\vzero&\text{in }\Omega_B\cup\Omega_F\setminus \Omega_B^i,\\
        \nabla\cdot\vv_{0,i}=0&\text{in }\Omega_B\cup\Omega_F,\\
        \jump{\vv_{0,i}}=\vzero&\text{on }\Gamma,\\
        \jump{\uT(\vv_{0,i},q_{0,i};\mu)\vn}=-2\jump{\mu}\uE(\vu_0)|_{\Omega_B^i}\vn&\text{on }\partial \Omega_B^i,\\
        \jump{\uT(\vv_{0,i},q_{0,i};\mu_F)\vn}=\vzero&\text{on }\Gamma\setminus\partial \Omega_B^i,
    \end{cases}
\end{equation}
for $i=1,\ldots,M$. Note that, as outlined above, we neglect surface tension ($\gamma=0$) since that can be accounted for exactly in the case of circular droplets.

In the above $M$ problems, we specify no boundary conditions on $\partial\Omega$. Although \cref{prop:correctbase} assumes homogeneous data there, we instead solve each problem through \cref{eq:singledroplet} (setting $\widetilde\gamma=0$ and $\vphi=\vu_0$), where the homogeneous Stokes equation is extended to $\bR^d\setminus \Omega_B^i$ and decay conditions as $\norm{\vx}\to\infty$ are enforced for closure. The benefits of this have been discussed in \cref{sec:singledroplet}.

\begin{remark}
    This decomposition into $M$ problems with a simply supported forcing term bears some similarity to a partition-of-unity strategy, in that global forcing terms are decomposed into $M$ local contributions at each droplet.
\end{remark}

The first correction is then defined by summing the $M$ local contributions,
\begin{equation*}
    \vv_0=\sum_{i=1}^M\vv_{0,i},
    \qquad
    q_0=\sum_{i=1}^Mq_{0,i},
\end{equation*}
yielding the updated intermediate approximation
\[
    \vu_{1/2}=\vu_0+\vv_0,
    \qquad
    p_{1/2}=p_0+\sum_{i=1}^M\frac{\gamma_i}{r_i}\chi_{\Omega_B^i}+q_0.
\]
(Note that the surface-tension-related pressure term discussed above has also been included.)

\subsection{Initial correction of boundary effects}
It is crucial to make several further observations concerning approximation $(\vu_{1/2},p_{1/2})$:
\begin{itemize}
    \item Each $(\vv_{0,i},q_{0,i})$ is computed based on the modified viscosity field \cref{eq:mupartial}, which is inconsistent with the true $\mu$ at all droplets except the $i$-th droplet. Consequently, if $M>1$, $(\vu_{1/2},p_{1/2})$ violates the Stokes equations at all droplet locations. This is made rigorous in \cref{prop:correctfull} below. Nevertheless, as we discuss in \cref{sec:termination}, we expect the intermediate approximation $(\vu_{1/2},p_{1/2})$ to be more accurate than the lowest-level surrogate $(\vu_0,p_0)$ at $\overline{\Omega_B}$, despite the partial lack of consistency.
    \item The homogeneous boundary conditions at $\partial\Omega$ are neglected in favor of the more computationally efficient decay conditions as $\norm{\vx}\to\infty$. As such, $(\vv_0,q_0)$ has support that includes the whole $\overline{\Omega}$, so that the new approximation $(\vu_{1/2},p_{1/2})$ violates the boundary conditions on $\partial\Omega$.
\end{itemize}

These observations behoove us to add further corrections to improve consistency and accuracy. In order to streamline our approach, we choose to follow up the newly added ``droplet correction'' $(\vv_0,q_0)$ with a ``boundary correction'' $(\vv_{1/2},q_{1/2})$, whose target is recovering consistency in the boundary conditions at $\partial\Omega$.

To derive a boundary-correction equation, we first identify the error in Dirichlet and Neumann boundary conditions:
\begin{equation*}
    \begin{cases}
        (\vu-\vu_{1/2})|_{\Gamma_D}=\vg-(\vu_0+\vv_0)|_{\Gamma_D}=-\vv_0|_{\Gamma_D}=:\vgcorr{1/2},\\
        (\uT(\vu,p;\mu)\vn-\uT(\vu_{1/2},p_{1/2};\mu)\vn)|_{\Gamma_N}=\ldots=-(\uT(\vv_0,q_0;\mu)\vn)|_{\Gamma_N}=:\vfcorr{1/2}.
    \end{cases}    
\end{equation*}
Adding any correction field $(\vv_{1/2},q_{1/2})$ satisfying the above boundary conditions would yield a surrogate that is consistent at the domain boundaries. In order to reduce the ``interior inconsistencies'' that such correction adds, we compute $(\vv_{1/2},q_{1/2})$ by solving the no-droplet Stokes problem \cref{eq:fem} with uniform viscosity $\mu_F$, replacing the original boundary conditions $\vg$ and $\vf$ with $\vgcorr{1/2}$ and $\vfcorr{1/2}$, respectively. We then obtain the updated approximation
\[
    \vu_1 := \vu_{1/2} + \vv_{1/2}, 
    \qquad
    p_1 := p_{1/2} + q_{1/2},
\]
which, by construction and linearity, satisfies the correct boundary conditions in the original problem \cref{eq:cutfem}, as well as the Stokes equation in the fluid domain $\Omega_F$ and continuity of the flow field across $\Gamma$.

\subsection{Further corrections and iteration}\label{sec:full:full}

Similarly to the lowest-fidelity approximation $(\vu_0,p_0)$, approximation $(\vu_1,p_1)$ attains the correct boundary conditions but does not satisfy the stress balance equation within the droplets $\Omega_B$ nor the interfacial force balance at $\Gamma=\partial\Omega_B$. At the same time, we expect the updated surrogate to be more accurate, e.g., in the sense that its residual with respect to the original PDE \cref{eq:cutfem} is smaller.

To further improve accuracy, we iterate the above steps, leading to the alternating correction scheme summarized in \cref{algo:full}  and illustrated schematically in \cref{fig:diagram}. Operators $\text{S\_boundary}$ and $\text{S\_droplet}$ are used to denote solution operators for problems \cref{eq:fem,eq:singledroplet}, respectively, with arbitrary forcings and BCs, as summarized in \cref{tab:operators}. In practice, the former is evaluated through finite elements, whereas the latter is either evaluated through discretized boundary-integral equations or as discussed in \cref{sec:basis}. Termination criteria are discussed in \cref{sec:termination}.

\begin{algorithm}[ht]
	\caption{Approximate solution of Stokes problem with droplets at $\Omega_B=\bigcup_{i=1}^MB(\vx_i;r_i)$}
	\begin{algorithmic}[1]\label{algo:full}
		\REQUIRE{viscosities $\mu_i=\mu(B(\vx_i;r_i))$ and surface tensions $\gamma_i=\gamma(\partial B(\vx_i;r_i))$}
		\REQUIRE{BCs $\vg$ (Dirichlet) and $\vf$ (Neumann)}
		\STATE{Compute baseline $(\vu_0,p_0)\gets\text{S\_boundary}(\vforce,\vg,\vf)$}
		\STATE{Update $p_0\gets p_0+\sum_{i=1}^M\gamma_i\chi_{B(\vx_i;r_i)}/r_i$ to account for surface tension}
		\STATE{Initialize $\vv_{-1,i}\gets\vzero$ for all $i=1,\ldots,M$, $\vv_{-1}\gets\vzero$, and $\vv_{-1/2}\gets\vu_0$}
		\FOR{$j=0,1,\ldots,n_{\text{iter}}-1$}
    		\STATE{\# \textit{compute corrections for droplet errors at $\overline{\Omega_B}$ to get $(\vu_{j+1/2},p_{j+1/2})$} \#}
    		\FOR{$i=1,\ldots,M$}
        		\STATE{Get interior consistency error field $\vphi_{j-1/2,i}\gets\vv_{j-1}+\vv_{j-1/2}-\vv_{j-1,i}$}\label{item:correction}
        		\STATE{Compute half-correction component $(\vv_{j,i}, q_{j,i})\gets\text{S\_droplet}(\mu_i,0,\vphi_{j-1/2,i},B(\vx_i;r_i))$}
    		\ENDFOR
    		\STATE{Compute interior half-correction $\vv_j\gets\sum_{i=1}^M\vv_{j,i}$ and $q_j\gets\sum_{i=1}^Mq_{j,i}$}
    		\STATE{Update $\vu_{j+1/2}\gets\vu_j+\vv_j$ and $p_{j+1/2}\gets p_j+q_j$}
    		\STATE{\# \textit{compute correction for boundary errors at $\partial\Omega$} to get $(\vu_{j+1},p_{j+1})$ \#}
    		\STATE{Get boundary consistency errors $\vgcorr{j+1/2}\gets-\vv_j|_{\Gamma_D}$ and $\vfcorr{j+1/2}\gets-(\uT(\vv_j,q_j;\mu)\vn)|_{\Gamma_N}$}
    		\STATE{Compute boundary half-correction $(\vv_{j+1/2},q_{j+1/2})\gets\text{S\_boundary}(\vzero,\vgcorr{j+1/2},\vfcorr{j+1/2})$}
    		\STATE{Update $\vu_{j+1}\gets\vu_{j+1/2}+\vv_{j+1/2}$ and $p_{j+1}\gets p_{j+1/2}+q_{j+1/2}$}
		\ENDFOR
        \RETURN{$(\vu_{n_{\text{iter}}},p_{n_{\text{iter}}})$}
	\end{algorithmic}
\end{algorithm}

\begin{figure}
    \centering
    \tikzexternaldisable
    \begin{tikzpicture}[-latex]
        \node[ellipse,draw] (start) at (-4.25,0) {$(\vu_0,p_0)$};
        \node[draw] (drop) at (0,0) {(approximate) correction at $\overline{\Omega_B}$};
        \path (start) edge (drop);
        \node[diamond,draw] (sum) at (0,-2) {\Large$+$};
        \path ([xshift=-20mm]drop.south) edge node[sloped,below] {\scriptsize$i=1$} (sum);
        \path ([xshift=-10mm]drop.south) edge node[sloped,below] {\scriptsize$i=2$} (sum);
        \node (dots) at (.25,-1) {\Large$\cdots$};
        \path ([xshift=20mm]drop.south) edge node[sloped,below] {\scriptsize$i=M$} (sum);
        \node[gray,text width=4cm,align=center] (solvedrop) at (-3.5,-1) {S\_droplet\\(solve \cref{eq:singledroplet} on $\Omega_B^i$)};
        \node[gray] (solvedropbrace) at (-2,-1) {$\Bigg\{$};
        \node[gray] (solvedropbracetwo) at (2,-1) {$\Bigg\}$};
        \node[ellipse,draw] (half) at (3,-2) {$(\vu_{j+1/2},p_{j+1/2})$};
        \path (sum) edge (half);
        \node[draw] (boundary) at (7.35,-2) {correction at $\partial\Omega$};
        \path (half) edge (boundary);
        \node[gray,text width=3cm,align=center] (solveboundary) at (8.45,-1.1) {S\_boundary\\(solve \cref{eq:fem})};
        \node[ellipse,draw] (full) at (7.35,0) {$(\vu_{j+1},p_{j+1})$};
        \path (boundary) edge (full);
        \node[diamond, aspect=3, inner sep=0pt,draw] (end) at (4.25,0) {converged?};
        \node[draw] (finish) at (4.25,-.95) {terminate};
        \path (full) edge (end);
        \path (end) edge node[above] {\scriptsize no} (drop);
        \path (end) edge node[left] {\scriptsize yes} (finish);
    \end{tikzpicture}
    \tikzexternalenable
    \caption{Alternating corrections in \cref{algo:full}, starting from the lowest-fidelity surrogate $(\vu_0,p_0)$.}
    \label{fig:diagram}
\end{figure}

\begin{table}[ht]
    \centering
    \begin{tabular}{c|c|c|c}
        Operator & Descr. & Domain of input(s) & Image spaces\\
        \hline
        $\text{S\_boundary}(\vforce,\vg,\vf)$ & solve \cref{eq:fem} & $[H^{-1}(\Omega)]^d\times[H^{1/2}(\Gamma_D)]^d\times[H^{-1/2}(\Gamma_N)]^d$ & $[H^1(\Omega)]^d\times L^2(\Omega)$\\
        $\text{S\_droplet}(\mu,\widetilde\gamma,\vphi,\Omega_B^i)$ & solve \cref{eq:singledroplet} & $\mathbb R\times\mathbb R\times[H^1(\Omega_B^i)]^d\times\{\Omega_B^i\}_i$ & $[H_{\text{loc}}^1(\mathbb R^d)]^d\times L_{\text{loc}}^2(\mathbb R^d)$\\
        $\mathcal S(\vphi)$ and $\mathcal S_N(\vphi)$ & eval.~\cref{eq:calS} & $[H^1(\Omega_B^\sharp)]^d$ & $[H_{\text{loc}}^1(\mathbb R^d)]^d\times L_{\text{loc}}^2(\mathbb R^d)$
    \end{tabular}
    \caption{PDE-solution operators used in this work.}
    \label{tab:operators}
\end{table}

Fundamentally, each iteration proceeds in the same way: $M$ solves of \cref{eq:singledroplet} (one per droplet) are followed by a single global solve of \cref{eq:fem} to restore boundary consistency. In particular, at the $j$-th iteration of \cref{algo:full}, we set up $M$ different one-droplet error equations \cref{eq:singledroplet} with certain forcing terms $\vphi=\vphi_{j-1/2,i}$, whose solutions are $(v_{j,i},q_{j,i})$, $i=1,\ldots,M$. As the following result shows, the forcing terms $\vphi_{j-1/2,i}$ are selected so that the partial correction $(\vu_j+\vv_{j,i},p_j+q_{j,i})$ is consistent with the original equation over the $i$-th droplet (although it is inconsistent at other droplets), generalizing \cref{prop:correctbase}.

\begin{proposition}\label{prop:correctfull}
    For any $j\ge0$, let $(\vu_j,p_j)$ denote the approximation produced by \cref{algo:full}. Define the errors $\vv^{(j)}=\vu-\vu_j$ and $q^{(j)}=p-p_j$. Then
    \begin{equation*}
        \begin{cases}
            -\nabla\cdot\uT(\vv^{(j)},q^{(j)};\mu)=2(\mu-\mu_F)\nabla\cdot\uE(\vphi_{j-1/2})&\text{in }\Omega_B\cup\Omega_F,\\
            \nabla\cdot\vv^{(j)}=0&\text{in }\Omega_B\cup\Omega_F,\\
            \jump{\vv^{(j)}}=\vzero&\text{on }\Gamma,\\
            \jump{\uT(\vv^{(j)},q^{(j)};\mu)\vn}=-2\jump{\mu}\uE(\vphi_{j-1/2})|_{\Omega_B}\vn&\text{on }\Gamma,\\
            \vv^{(j)}=\vzero&\text{on }\Gamma_D,\\
            \uT(\vv^{(j)},q^{(j)};\mu)\vn=\vzero&\text{on }\Gamma_N,
        \end{cases}
    \end{equation*}
    where
    \[
        \vphi_{j-1/2}=\sum_{i=1}^M\sum_{i'\neq i}\vv_{j-1,i'}\chi_{\Omega_B^i}+\vv_{j-1/2}.
    \]
\end{proposition}
\begin{proof}
    The claim can be shown by induction on $j$, through careful use of the linearity of problems \cref{eq:cutfem,eq:fem,eq:singledroplet}. See \cref{app:proof} for the full proof.
\end{proof}

\begin{remark}\label{rem:phi}
    Note that, for all $j$ and $i=1,\ldots,M$,
    \[
    \vphi_{j-1/2}|_{\Omega_B^i}=\sum_{i'\neq i}\vv_{j-1,i'}+\vv_{j-1/2}=\vv_{j-1}-\vv_{j-1,i}+\vv_{j-1/2}=\vphi_{j-1/2,i},
    \]
    justifying line \ref{item:correction} in \cref{algo:full}.
\end{remark}

As already outlined in \cref{sec:preliminaries}, the error equation provided in \cref{prop:correctfull} is paramount in our approach, since it gives conditions to be satisfied (in approximate form) by further corrections. We also note the following.
\begin{remark}\label{rem:stagger}
    For simplicity, we alternate ``droplet corrections'' and ``boundary corrections'', with each forming a half-step in our iterative procedure. One might consider making several droplet corrections in a row, thus making the process staggered. (Note that repeated boundary corrections would not be useful because full-step surrogates attain boundary consistency.) This option might be useful as a way to achieve a more balanced procedure when droplet errors dominate boundary errors (in some sense). We leave analysis and testing of this option for future work.
\end{remark}

Concerning the computational efficiency of \cref{algo:full}, it is important to note that the proposed procedure involves repeated solves of \cref{eq:fem} and \cref{eq:singledroplet} with varying right-hand sides. Especially when these problems are discretized with high resolution, efficiency can be greatly improved by preassembling and pre-factorizing (e.g., via sparse LU) the corresponding stiffness matrices. This is particularly beneficial for \cref{eq:fem}. For \cref{eq:singledroplet}, an even cheaper approximate strategy can be adopted, potentially avoiding repeated linear solves altogether at the cost of some accuracy. This alternative is discussed next.

\subsection{Improving efficiency by exploiting self-similarity}\label{sec:basis}

We now discuss how to accelerate the computation of the single-droplet problems~\cref{eq:singledroplet}  in free space. This is crucial for the efficiency of~\cref{algo:full}, since one such problem must be solved for each droplet at every iteration.  The approach described below relies on the geometric similarity of the droplets, meaning that each droplet can be mapped to a common reference configuration by an affine transformation. Under the present assumption that all droplets are circular, this property holds exactly. Extensions to non-circular droplets, as may arise in time-dependent simulations, are discussed in~\cref{sec:extensions}.

In addition, we will assume that the surface tension is $\widetilde\gamma=0$ for simplicity (as discussed at the beginning of \cref{sec:surrogate}, this assumption is natural for circular droplets), although our presentation naturally generalizes to non-trivial surface tensions.

By geometric similarity of the droplets, the $M$ single-droplet problems \cref{eq:singledroplet} (one per droplet) may be remapped to a reference configuration, viz., the unit disk $\Omega_B^\sharp=B(\vzero;1)$, through droplet-dependent affine mappings. See \cref{app:change} for more details and refer to \cref{sec:singledroplet} for a discussion on how this follows from our choice of an unbounded domain. As such, the task of (approximately) solving \cref{eq:singledroplet} on $M$ different domains and forcing terms can then be replaced by the task of solving \cref{eq:singledroplet} on the unit disk with the $M$ different forcing terms resulting from the remappings. Equivalently, we must evaluate operator
\begin{equation}\label{eq:calS}
    \mathcal S:[H^1(\Omega_B^\sharp)]^d\ni\vphi\mapsto(\vu,p)\text{ solving \cref{eq:singledroplet} with $\Omega_B^i=\Omega_B^\sharp$ and flow field $\vphi$}
\end{equation}
at $M$ different inputs $\{\widetilde\vphi_{j-1/2,i}\}_{i=1}^M$ (each obtained by remapping $\vphi_{j-1/2,i}$ from \cref{algo:full}), one per droplet, at each iteration.

\begin{remark}\label{rem:muuniform}
    Operator $\mathcal S$ depends also on the viscosity values $\mu(\vx_i)$ and $\mu_F$, although this dependence is kept implicit in our notation. This fact is of relevance in multi($>2$)-phase flows, where different droplets have different viscosities: $\mu|_{\Omega_B^i}\not\equiv\mu|_{\Omega_B^1}$. In such cases, the upcoming construction must be repeated for each viscosity value.
\end{remark}

A cheap-to-evaluate surrogate for $\mathcal S$ can be built through a \emph{(linear) subspace approach}: given a set of $N$ linearly independent basis functions $\{\widehat\vphi_k\}_{k=1}^N\subset[H^1(\Omega_B^\sharp)]^d$, for any $\vphi\in[H^1(\Omega_B^\sharp)]^d$ we approximate $\mathcal S(\vphi)\approx\mathcal S(\mathcal P_N\vphi)$, with $\mathcal P_N:[H^1(\Omega_B^\sharp)]^d\to\text{span}\{\widehat\vphi_k\}_{k=1}^N$ a suitable projector operator.

Computational efficiency can be achieved by exploiting the linearity of $\mathcal S$.  
We introduce the 
coordinate map $\mathcal Q_N:[H^1(\Omega_B^\sharp)]^d\to\bR^N$, defined as $\sum_{k=1}^N(\mathcal Q_N\vphi)_k\widehat\vphi_k:=\mathcal P_N\vphi$. Let $\{\mathcal S(\widehat\vphi_k)\}_{k=1}^N$ denote the corresponding precomputed single-droplet solutions. We have
\begin{equation}\label{eq:basiss}
    \vphi\approx\mathcal P_N\vphi=\sum_{k=1}^N(\mathcal Q_N\vphi)_k\widehat\vphi_k\quad\Rightarrow\quad\mathcal S(\vphi)\approx\mathcal S_N(\vphi):=\mathcal S(\mathcal P_N\vphi)=\sum_{k=1}^N(\mathcal Q_N\vphi)_k\mathcal S(\widehat\vphi_k).
\end{equation}
This means that, after precomputing the single-droplet  solutions, evaluating $\mathcal S_N(\vphi)$ comes at the only cost of evaluating the coordinate map $\mathcal Q_N$ and a linear combination of single-droplet solutions.

\begin{remark}\label{rem:polynomials}
    In practice, we choose the sampled flow fields $\{\widehat\vphi_k\}_{k=1}^N$ to have polynomial components up to a prescribed degree $D\geq0$. By proxy, we refer to the associated single-droplet solution $\mathcal S(\widehat\vphi_k)$ as having the same degree as $\widehat\vphi_k$. This has the advantage of implementation simplicity, as well as the guarantee of spectral approximation convergence for spatially smooth fields $\vphi$, although alternatives may be considered. Moreover, it allows evaluating the projector $\mathcal P_N$ and the coordinate map $\mathcal Q_N$ through simple (least-squares) polynomial interpolation. More details on this are provided in \cref{app:poly} for $d=2$ dimensions. In particular, see \cref{tab:basis} in \cref{app:poly} for the basis functions $\{\widehat{\phi}_k\}_{k=1}^N$ used in our numerical experiments, up to cubic degree ($D=3$).
\end{remark}

This strategy employs an \emph{offline-online} decomposition: a large(r) cost is incurred in some precomputation so that later numerical operations may be made cheap(er). Namely, let $c_{\text{high-fi}}$ denote the cost of a single high-fidelity evaluation of $\mathcal S$ through CutFEM or boundary-integral-based method, and assume that evaluating $\mathcal S_N$ has $\mathcal O(N^3)\ll c_{\text{high-fi}}$ cost because of the coordinate map $\mathcal Q_N$. If $N_{\text{eval}}$ evaluations of $\mathcal S$ at different inputs are required, then an $\mathcal O(N_{\text{eval}}c_{\text{high-fi}})$ direct-evaluation cost is replaced by an $\mathcal O(Nc_{\text{high-fi}}+N_{\text{eval}}N^3)$ approximate-evaluation cost.

\begin{figure}
    \centering
    \begin{tabular}{ccc}
        $\widehat\vphi_1(\vx)=\begin{pmatrix}x_2\\x_1\end{pmatrix}$ & $\widehat\vphi_5(\vx)=\displaystyle\frac12\begin{pmatrix}-2x_1x_2\\3x_1^2+x_2^2\end{pmatrix}$ & $\widehat\vphi_7(\vx)=\displaystyle\frac16\begin{pmatrix}3x_1^2x_2-x_2^3\\x_1^3-3x_1x_2^2\end{pmatrix}$\\[3mm]
        \begin{overpic}[percent,trim={3mm 0 375mm 0},clip,height=3.2cm]{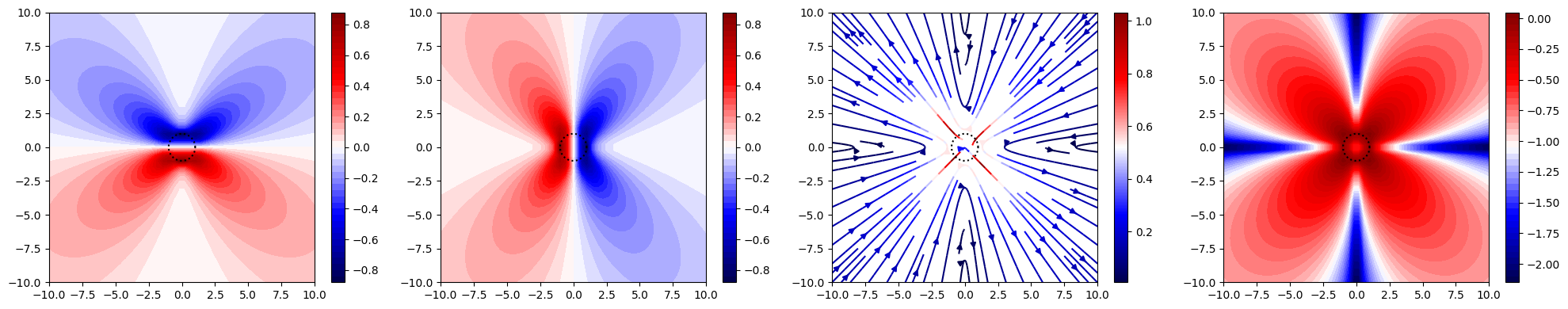}
            \put(42,-4){$x_1$}
            \put(-8,40){$x_2$}
        \end{overpic}&
        \phantom{ol}\begin{overpic}[percent,trim={3mm 0 375mm 0},clip,height=3.2cm]{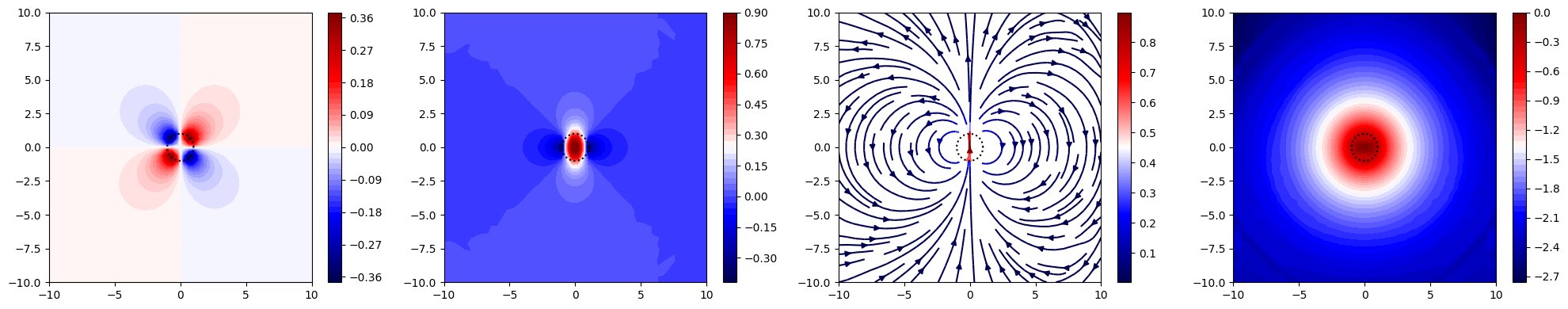}
            \put(41.5,-4){$x_1$}
            \put(-8,40){$x_2$}
        \end{overpic}&
        \phantom{o}\begin{overpic}[percent,trim={3mm 0 375mm 0},clip,height=3.2cm]{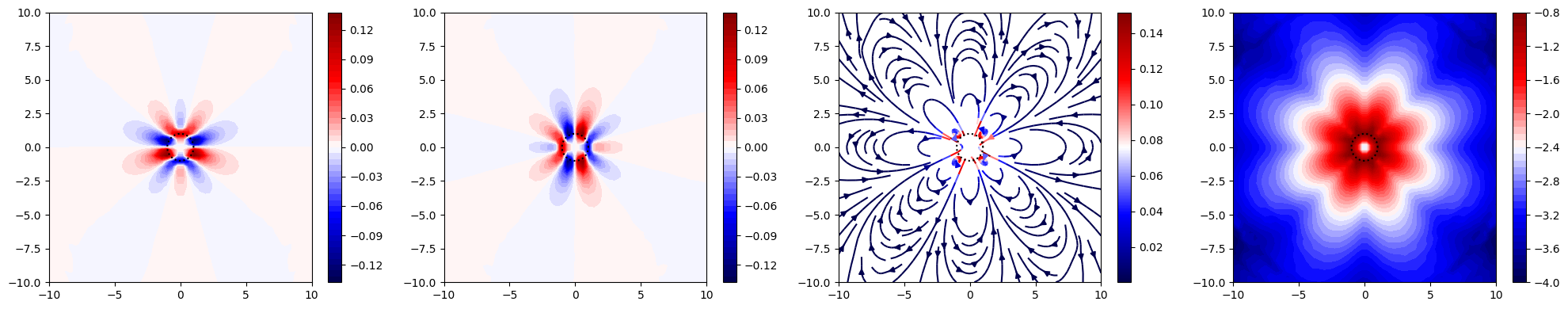}
            \put(41,-4){$x_1$}
            \put(-8,40){$x_2$}
        \end{overpic}
    \end{tabular}
    \caption{Qualitative plot of three of the precomputed single-droplet solutions corresponding to linear (left), quadratic (center), and cubic (right) flow fields $\widehat\vphi_k$, reported on top of the respective plots. (See \cref{tab:basis} for a complete list of the flow fields $\widehat\vphi_k$ forming a polynomial basis up to cubic degree.) The three displayed fields form a subset of the cubic basis described in \cref{app:poly}. We only plot the first components $(\widehat\vv_k)_1$ of the flow fields. In all cases, $\mu_F=1$ and $\mu|_{\Omega_B^\sharp}\equiv100$.}
    \label{fig:basis}
\end{figure}

Of course, an error is incurred due to the truncation of the basis at index $N<\infty$. It is relatively straightforward to derive \emph{a priori} bounds on the basis-truncation error. Specifically, since \cref{eq:singledroplet} is well-posed \cite{ladyzhenskaya1963,Pozrikidis1992}, for any bounded $U\subset\bR^d$ we can find a $\vphi$-independent constant $C$ such that
\begin{equation*}
    \norm{\mathcal S(\vphi)-\mathcal S_N(\vphi)}_{H^1(U)\times L^2(U)}\leq C\big|\mu(\vzero)-\mu_F\big|\left(\norm{\nabla\cdot\uE(\vphi-\vphi_N)}_{H^{-1}(\Omega_B^\sharp)} + \norm{\uE(\vphi-\vphi_N)\vn}_{H^{-1/2}(\partial \Omega_B^\sharp)}\right).
\end{equation*}
by linearity of $\mathcal S$. As such, we can control the error in the approximation of $\mathcal S(\vphi)$ by lowering the subspace-projection error. In our experiments, we found that a small $N\lesssim10$ is enough to achieve good accuracy in $d=2$ dimensions. Since, in \cref{algo:full}, $N_{\text{eval}}=n_{\text{iter}}M$ may be much larger than such $N$, our approach is justified. See also \cref{thm:convergenceD} for further discussion on how the subspace error affects the surrogate.

At the same time, we note that the suggested inexact evaluation of $\text{S\_droplet}$ in \cref{algo:full} invalidates the two non-homogeneous error equations in \cref{prop:correctfull}, which now read\footnote{This can be shown by induction, following the same steps as in \cref{app:proof}.}
\begin{equation}\label{eq:correctfullD}
    \begin{cases}
        -\nabla\cdot\uT(\vv^{(j)},q^{(j)};\mu)=2(\mu-\mu_F)\nabla\cdot\uE\left(\vphi_{j-1/2}+(\mathcal I-\mathcal P_{N,i})\sum_{j'=0}^{j-1}\vphi_{j'-1/2}\right)&\text{in }\Omega_B^i,i=1,\ldots,M,\\
        \jump{\uT(\vv^{(j)},q^{(j)};\mu)\vn}=-2\jump{\mu}\uE\left(\vphi_{j-1/2}+(\mathcal I-\mathcal P_{N,i})\sum_{j'=0}^{j-1}\vphi_{j'-1/2}\right)\Big|_{\Omega_B}\vn&\text{on }\partial\Omega_B^i,i=1,\ldots,M,\\
    \end{cases}
\end{equation}
having introduced the identity operator $\mathcal I$, as well as the operator $\mathcal P_{N,i}$ that (through affine remapping of $\mathcal P_N$) computes the subspace projection of a field over $\Omega_B^i$. This shows that off-subspace projection errors accumulate throughout \cref{algo:full}. At the same time, since $\mathcal P_N$ is assumed idempotent, $\mathcal P_{N,i}(\mathcal I-\mathcal P_{N,i})=0$, and the subspace-projected error equations are still those in \cref{prop:correctfull}. As such, replacing $\mathcal S$ with $\mathcal S_N$ requires no further modifications to our method.

As a final point, we wish to comment on how one may precompute the $N$ single-droplet solutions, corresponding to setting $\vphi=\widehat\vphi_k$ in \cref{eq:singledroplet} for $k=1,\ldots,N$. In our numerical tests, we compute these through a boundary-integral formulation, which is natural due to the unboundedness of the domain. Alternatively, finite-elements or finite-differences techniques on truncated domains might also be suitable. In \cref{fig:basis} we display some sample single-droplet solutions corresponding to the polynomial basis derived in \cref{app:poly}. The overall offline-online computational pipeline is summarized in \cref{algo:offline}.

\begin{algorithm}[ht]
    \renewcommand{\algorithmicensure}{}
	\caption{Offline-online decomposition}
	\begin{algorithmic}\label{algo:offline}
        \ENSURE{\hspace{-3.5mm}\# \textit{offline phase} \#}
		\REQUIRE{reference droplet shape $\Omega_B^\sharp$ (e.g., unit disk $B(\vzero;1)$), viscosities $\mu=\smash{\mu|_{\Omega_B^\sharp}}$ and $\mu_F$}
        \REQUIRE{if $\Omega_B^\sharp$ is \emph{not} circular, surface tension $\gamma=\gamma(\partial\Omega_B^\sharp)$}
        \REQUIRE{flow-field basis $\{\widehat\vphi_k\}_{k=1}^N$ or polynomial degree $D$}
        \STATE{If flow-field basis $\{\widehat\vphi_k\}_{k=1}^N$ is not given, assemble it with polynomials of degree$\;\leq D$; see \cref{rem:polynomials,tab:basis}}
		\FOR{$k=1,\ldots,N$}
    		\STATE{Compute $\mathcal S(\widehat\vphi_k)$ by solving \cref{eq:singledroplet} with $\Omega_B^i=\Omega_B^\sharp$ and flow field $\widehat\vphi_k$, and store the result to disk}
		\ENDFOR
        \ENSURE{\hspace{-3.5mm}\# \textit{online phase} \#}
		\REQUIRE{domain geometry $\Omega=\overline{\Omega_F}\cup\bigcup_{i=1}^M\Omega_B^i$ and BCs $\vg$ (Dirichlet) and $\vf$ (Neumann)}
		\REQUIRE{surface tensions $\gamma_i=\gamma(\partial\Omega_B^i)$}
		\STATE{Apply \cref{algo:full} with $\mathcal S_N$ as approximate solver in $\text{S\_droplet}$}
	\end{algorithmic}
\end{algorithm}

In particular, note that the offline computations are based on fixed reference shapes $\Omega_B^\sharp$ and viscosity ratios $\mu|_{\Omega_B^\sharp}/\mu_F$. In our tests we always use a single, circular reference shape, but it is possible to build a database of single-droplet solutions corresponding to a set of different reference shapes. The results of the offline phase (the single-droplet solutions) can be used online in arbitrary flow conditions (arbitrary $\Omega$ and boundary conditions) with arbitrary numbers $M$ of droplets, although it must be assumed that all droplets are geometrically similar to one of the reference shapes $\Omega_B^\sharp$. As mentioned at the start of \cref{sec:surrogate}, changes in surface tension may require re-running the offline phase if the droplet shape is not circular.

\section{Convergence analysis}\label{sec:termination}
The iterative scheme of \cref{algo:full} interleaves $M$ local \cref{eq:singledroplet} solvers with a global \cref{eq:fem} solver. The problems solved at different iterations differ only in their right-hand sides, i.e., the forcing terms on $\overline{\Omega_B^i}$ in \cref{eq:singledroplet} and the boundary conditions on $\partial\Omega$ in \cref{eq:fem}. The magnitudes of such right-hand sides, namely, the \emph{droplet residual norm}
\begin{equation}\label{eq:fixedfull}
    R_j=\norm{2(\mu-\mu_F)\nabla\cdot\uE(\vphi_{j-1/2})}_{H^{-1}(\Omega_B)} + \norm{2\jump{\mu}\uE(\vphi_{j-1/2})|_{\Omega_B}\vn}_{H^{-1/2}(\Gamma)}
\end{equation}
at the $j$-th ``droplet half-step'' and the \emph{boundary residual norm}
\begin{equation}\label{eq:fixedhalf}
    R_{j+1/2}=\norm{\vv_j}_{H^{1/2}(\Gamma_D)} + \norm{\uT(\vv_j,q_j;\mu)\vn}_{H^{-1/2}(\Gamma_N)}
\end{equation}
at the $j$-th ``boundary half-step'', $j=0,1,\ldots$, are useful to study the convergence of the iterates and also to set termination conditions for the iterations.

We first analyze the convergence of \cref{algo:full} under the assumption that the single-droplet problems are solved exactly, rather than with the subspace method in \cref{algo:offline}, whose effects on convergence are discussed later on in \cref{thm:convergenceD}. In the following, we focus on the flow field $\vu$ and its corrections $\vv$ in $H^1(\Omega)$, although the same derivation can be carried out for the pressure $p$ and its corrections $q$ in $L^2(\Omega)\setminus\bR$.

First, by the well-posedness of the Stokes boundary-value problem \cref{eq:fem} \cite{ladyzhenskaya1963,Pozrikidis1992}, the boundary correction $\vv_{j+1/2}$ may be bounded in terms of residual $R_{j+1/2}$:
\begin{equation}\label{eq:femwellposed}
    \norm{\vv_{j+1/2}}_{H^1(\Omega)}\leq C_{\text{BVP}}\left(\norm{\vv_j}_{H^{1/2}(\Gamma_D)} + \norm{\uT(\vv_j,q_j;\mu)\vn}_{H^{-1/2}(\Gamma_N)}\right)=C_{\text{BVP}}R_{j+1/2},
\end{equation}
with $C_{\text{BVP}}$ depending only on domain geometry ($\Omega$, $\Gamma_D$, and $\Gamma_N$) and $\mu_F$.

Let $\textup{dist}(A,B)=\inf_{(x,y)\in A\times B}\norm{x-y}$ denote the Euclidean distance between sets. Under a gap condition
\begin{equation}\label{eq:gap}
    \delta:=\min\left\{\min_{1\leq i<i'\leq M}\textup{dist}(\partial\Omega_B^i,\partial\Omega_B^{i'}),\textup{dist}(\Gamma,\partial\Omega)\right\}>0,
\end{equation}
which is necessary to ensure no overlap between interfaces and domain boundaries (with incompatible conditions), a similar relation can be derived also for the single-drop problem \cref{eq:singledroplet}:
\begin{equation}\label{eq:singlewellposedfull}
    \norm{\vv_j}_{H^1(\Omega)}=\norm{\sum_{i=1}^M\vv_{j,i}}_{H^1(\Omega)}\leq\sum_{i=1}^M\norm{\vv_{j,i}}_{H^1(\Omega)}\leq\sum_{i=1}^MC_{\text{drop}}R_{j,i}=C_{\text{drop}}R_j,
\end{equation}
where the local residuals $R_{j,i}$ are defined as
\begin{equation}\label{eq:fixedfullloc}
    R_{j,i}=\norm{2(\mu-\mu_F)\nabla\cdot\uE(\vphi_{j-1/2})|_{\Omega_B^i}}_{H^{-1}(\Omega_B^i)} + \norm{2\jump{\mu}\uE(\vphi_{j-1/2})|_{\Omega_B}\vn}_{H^{-1/2}(\partial\Omega_B^i)}.
\end{equation}
The constant $C_{\text{drop}}$ depends only on domain geometry ($\Omega_B$, $\Omega_F$, $\Gamma_D$, and $\Gamma_N$) and $\mu$.

In addition, the solution of the single-droplet problem \cref{eq:singledroplet} displays a decay as one moves away from the droplet, as also showcased in \cref{fig:basis}. This allows deriving some contractive properties on certain restrictions of correction $\vv_j$.
\begin{lemma}\label{lem:singlewellposed}
    Assume the gap condition \cref{eq:gap} and that $\textup{S\_droplet}$ is evaluated exactly within \cref{algo:full}, cf.~\cref{sec:basis}. With the notation from \cref{algo:full}, one has
    \begin{equation}\label{eq:singlewellposedboundary}
        R_{j+1/2}=\norm{\vv_j}_{H^{1/2}(\Gamma_D)}+\norm{\uT(\vv_j,q_j;\mu)\vn}_{H^{-1/2}(\Gamma_N)}\leq C_{\Omega_B\to\partial\Omega}R_j,
    \end{equation}
    and, for all $i\neq i'$,
    \begin{equation}\label{eq:singlewellposedvolume}
        \norm{\vv_{j,i}}_{H^1(\Omega_B^{i'})}+\norm{\uT(\vv_{j,i},q_{j,i};\mu)|_{\Omega_B}\vn}_{H^{-1/2}(\partial\Omega_B^{i'})}\leq C_{\Omega_B^i\to\Omega_B^{i'}}R_{j,i},
    \end{equation}
    cf.~\cref{eq:fixedfullloc}. The constants $C_{\Omega_B\to\partial\Omega}$ and $C_{\Omega_B^i\to\Omega_B^{i'}}$ depend only on domain geometry ($\Omega_B$ and $\Omega_F$) and $\mu$.
    
    For any $L>0$ there exists a minimal gap $\delta_\star(L)$, depending on $\Omega$ and on the shapes of the droplets, such that, if $\delta>\delta_\star(L)$, then $C_{\Omega_B\to\partial\Omega}<1/L$ and $C_{\Omega_B^i\to\Omega_B^{i'}}<1/L$ for all $i$ and $i'$. Such $\delta_\star(L)$ scales linearly with droplet size: $\delta_\star(L)\lesssim\max_i|\Omega_B^i|$.
\end{lemma}

\begin{proof}
    See \cref{app:prooftwo}.
\end{proof}

We can now prove convergence of our method under a gap condition.
\begin{theorem}\label{thm:convergence}
    Assume that $\textup{S\_droplet}$ is evaluated exactly within \cref{algo:full}, cf.~\cref{sec:basis}. There exists a minimal gap $\delta^\star$, depending on $\Omega$ and on the shapes of the droplets, such that, if $\delta>\delta^\star$, $\vu_j$ from \cref{algo:full} converges to $\vu$ in the $H^1$-norm as $j\to\infty$.
\end{theorem}
\begin{proof}
    We first wish to show that, for large enough $\delta$, $R_{j+1}\leq\rho R_j$ for all $j$, with $\rho<1$, so that $R_j\leq\rho^jR_0\to0$. To this aim, we use the shorthand
    \begin{equation*}
        R^{(i)}(\vphi):=\norm{\nabla\cdot\uE(\vphi)}_{H^{-1}(\Omega_B^i)}+\norm{\uE(\vphi)|_{\Omega_B}\vn}_{H^{-1/2}(\partial\Omega_B^i)},
    \end{equation*}
    which, by standard elliptic-regularity arguments, is equivalent to the Sobolev $H_0^1(\Omega_B^i)$-norm of $\vphi$, to derive
    \begin{equation*}
        R_{j+1}=\sum_{i=1}^MR_{j+1,i}=\sum_{i=1}^M2\abs{\mu|_{\Omega_B^i}-\mu_F}R^{(i)}(\vphi_{j+1/2})\leq C_\mu\left(\sum_{i=1}^M\sum_{i'\neq i}R^{(i)}(\vv_{j,i'})+\sum_{i=1}^MR^{(i)}(\vv_{j+1/2})\right),
    \end{equation*}
    with the notation $C_\mu=2\max_i|\mu|_{\Omega_B^i}-\mu_F|$. Note that the last step uses the definition of $\vphi_{j+1/2}$, cf.~\cref{prop:correctfull}.
    
    By \cref{eq:femwellposed,eq:singlewellposedboundary,eq:singlewellposedvolume} we now obtain
    \begin{equation*}
        R_{j+1}\leq C_\mu\left(\sum_{i=1}^M\sum_{i'\neq i}C_{\Omega_B^i\to\Omega_B^{i'}}R_{j,i}+C_{\text{BVP}}R_{j+1/2}\right)\leq C_\mu\left(M\left(\max_{i,i'\neq i}C_{\Omega_B^i\to\Omega_B^{i'}}\right)+C_{\text{BVP}}C_{\Omega_B\to\partial\Omega}\right)R_j.
    \end{equation*}
    By \cref{lem:singlewellposed}, we may choose a minimal gap $\delta^\star:=\delta_\star\left(C_\mu\left(M+C_{\text{BVP}}\right)\right)$ that makes the constant in the above bound less than one.
    
    Since $R_j\to0$, by \cref{prop:correctfull} and the well-posedness of the interface Stokes equation \cref{eq:cutfem}, we obtain convergence: $\norm{\vu_j-\vu}_{H^1(\Omega)}=\norm{\vv^{(j)}}_{H^1(\Omega)}\leq C_{\text{well-posedness}}R_j\to0$.
\end{proof}

\begin{remark}\label{rem:traditional}
    The above result provides an \emph{a priori} sufficient condition for convergence only for large-enough gap $\delta$. It is unclear at this point if this is a limitation of the theory or if any pathological setup may exist where $0<\delta<\delta^\star$ and \cref{algo:full} does not converge. Intuitively, we expect the contractive properties of our fixed-point iterations to degrade when $\delta$ is small, and to eventually fail completely when $\delta=0$. In \cref{sec:numerics} we include some numerical experiments to test this, confirming that contractivity seems linked to $\delta>0$, although our method converged (with a non-monotonical residual decrease) even when droplets touch or even overlap ($\delta=0$). Interestingly, our tests show that small boundary gaps $\textup{dist}(\Gamma,\partial\Omega)$ are a much bigger detriment to convergence than comparably small droplet gaps $\min_{1\leq i<i'\leq M}\textup{dist}(\partial\Omega_B^i,\partial\Omega_B^{i'})$.
    
    At the same time, we also note that settings with small or vanishing gaps are difficult to handle even more so with standard approximation techniques. For instance, finite-element-based approaches require mesh refinements, whereas boundary-integral formulations must see a local increase in the number of quadrature points. We expect our approach to deal with such cases (through iteration) more effectively than alternative options do (through refinement).
\end{remark}

On the practical side, we require a convergence indicator as stopping criterion in \cref{algo:full}. For this we employ a simplified version of the boundary residual \cref{eq:fixedhalf}, with the $L^\infty$-norm replacing the expensive $H^{\pm1/2}$ ones. This is cheaply computable and readily available, since evaluations of the boundary residual are anyway needed when solving \cref{eq:fem} within \cref{algo:full}. More robust, but costly, alternative termination criteria may also be based on global $L^2(\Omega)$- or $H^1(\Omega)$-norms of the full increments $\vu_{j+1}-\vu_j=\vv_j+\vv_{j+1/2}$.

\begin{remark}\label{rem:relaxation}
    If the chosen convergence indicator provides evidence of diverging behavior (which \emph{a priori} requires very small or vanishing gaps), it remains possible to introduce a relaxation parameter $0<\omega<1$, as is typical in iterative (e.g., Richardson) schemes \cite{hackbusch_iterative_2016}. In \cref{algo:full}, a Richardson relaxation can be achieved through a simple rescaling\footnote{This not-so-trivial property holds true due to two combined properties: (i) the linearity of the $(M+1)$ PDEs solved at each iteration and (ii) the exactness of all full-step approximations $(\vu_j,p_j)$ in terms of boundary conditions.} of all fields $\vphi_{j,i}$ by the factor $\omega$. While potentially improving the methods' stability, this strategy may obviously lead to reduced convergence speeds if $\omega$ is chosen too small. In our numerical experiments in \cref{sec:numerics}, relaxation was never needed.
\end{remark}

The above discussion and results assume that problem \cref{eq:singledroplet} is solved exactly, e.g., through boundary-integral operators, cf.~\cref{app:prooftwo}, rather than through the subspace-projection approach described in \cref{sec:basis}, which introduces approximation errors. We have already discussed how these inaccuracies affect the approximation error in a cumulative way through \cref{eq:correctfullD}. Our convergence theory can be extended to this case as well.

First, we must note that, if the subspace approach is applied, then the droplet residual $R_j$ in \cref{eq:fixedfull} is not the exact residual of surrogate $(\vu_j,p_j)$ with respect to the target PDE (in the sense of \cref{prop:correctfull}). Indeed, all droplet corrections $\vv_j$ at prior iterations have been computed disregarding components within the kernel of $\mathcal P_N$. For this reason, $R_j$ in \cref{eq:fixedfull} should be seen as the magnitude of the residual \emph{after projection onto the range of $\mathcal P_N$}. The full residual is given in \cref{eq:correctfullD}, notably including also an off-subspace term of the form
\begin{equation*}
    (\mathcal I-\mathcal P_{N,i})\sum_{j'=0}^{j-1}\vphi_{j'-1/2}=(\mathcal I-\mathcal P_{N,i})\sum_{j'=0}^{j-1}\left(\sum_{i'\neq i}\vv_{j'-1,i'}+\vv_{j'-1/2}\right)=(\mathcal I-\mathcal P_{N,i})\left(\vu_j-\sum_{j'=0}^{j-1}\vv_{j'-1,i}\right)
\end{equation*}
on droplet $i=1,\ldots,M$, at iteration $j$. We can easily see that $R_j$ in \cref{eq:fixedfull} may be obtained from the actual residual in \cref{eq:correctfullD} by ignoring the above off-subspace term.

These off-subspace residual components are effectively invisible to our surrogate model since they belong to the kernel of $\mathcal P_N$. In particular, if \cref{algo:full} converges to $\lim_{j\to\infty}(\vu_j,p_j)=(\widetilde\vu,\widetilde p)$ in $H^1(\Omega)\times L^2(\Omega)$, then $(\widetilde\vu,\widetilde p)$ is generally not the exact solution $(\vu,p)$. Rather, it is a ``subspace-projected'' version of the exact solution, in the sense that its residual belongs to the kernel of $\mathcal P_N$: by applying a limiting argument to \cref{eq:correctfullD} we obtain
\begin{equation}\label{eq:correctfulllimD}
    \begin{cases}
        -\nabla\cdot\uT(\widetilde\vu-\vu,\widetilde p-p;\mu)=2(\mu-\mu_F)\nabla\cdot\uE\left((\mathcal I-\mathcal P_{N,i})\widetilde\vphi\right)&\text{in }\Omega_B^i,i=1,\ldots,M,\\
        \jump{\uT(\widetilde\vu-\vu,\widetilde p-p;\mu)\vn}=-2\jump{\mu}\uE\left((\mathcal I-\mathcal P_{N,i})\widetilde\vphi\right)\Big|_{\Omega_B}\vn&\text{on }\partial\Omega_B^i,i=1,\ldots,M,\\
    \end{cases}
\end{equation}
with $\widetilde\vphi=\sum_{j=0}^\infty\vphi_{j-1/2}\in[H^1(\Omega_B)]^d$. This being said, it remains to prove that \cref{thm:convergence} generalizes to this case: under a gap condition, the iterations in \cref{algo:full} converge, i.e., the above $(\widetilde\vu,\widetilde p)$ is well-defined, even when \cref{eq:singledroplet} is solved inexactly through subspace projection. The following result guarantees this.

\begin{theorem}\label{thm:convergenceD}
    Assume that, within \cref{algo:full}, $\textup{S\_droplet}$ is evaluated through the subspace-projection procedure from \cref{sec:basis}, with $\mathcal P_N$ a projection from $[H^1(\Omega_B)]^d$ to $\mathcal P_N([H^1(\Omega_B)]^d)\subset[H^1(\Omega_B)]^d$. There exists a minimal gap $\delta^\star$, depending on $\Omega$ and on the shapes of the droplets, such that, if $\delta>\delta^\star$, $\vu_j$ from \cref{algo:full} converges to $\widetilde\vu$ in the $H^1$-norm as $j\to\infty$. The limit $\widetilde\vu$ satisfies \cref{eq:correctfulllimD}.
\end{theorem}
\begin{proof}
    Leveraging the boundedness of the subspace-projection operators $\mathcal P_{N,i}$, the arguments from the proof of \cref{thm:convergence} can be applied here as well. See \cref{app:proofthree} for the full derivation.
\end{proof}

An immediate consequence of the above result is that choosing a subspace projector $\mathcal P_N$ with good approximation properties on the family of fields $\{\vphi_{j-1/2}\}_{j=0}^\infty$ is crucial for the accuracy of our method. This is numerically verified in \cref{sec:numerics:conv}.

\section{Complexity}\label{sec:timing}

We now discuss the computational complexity of the proposed surrogate model, both at training and evaluation time.  Let $M$ denote the number of droplets and $N$ the number of single-droplet solutions used to approximate the action of the single-droplet solution operator $\mathcal S$ in \cref{eq:basiss}. To streamline our discussion, we will assume that $M>N$, since (i) our numerical tests show that a small number $N$ of precomputed single-droplet solutions is usually sufficient to achieve modest accuracy (say, $10^{-4}$), cf.~the theoretical discussion in \cref{sec:basis} as well as our numerical tests in \cref{sec:numerics}, and (ii) the many-droplet regime is ultimately our target.

We specialize our discussion to two useful computational pipelines that use \cref{algo:full,eq:basiss} to emulate \cref{eq:cutfem}. They apply in different scenarios, depending on the target application.

\subsection{Pipeline 1: static emulation of the solution field}\label{sec:field}

A first setting where our surrogate model is useful is when one aims to approximately compute the stationary solution of the Stokes interface problem \cref{eq:cutfem}. This is relevant, for instance, when investigating the effect of the droplets $\Omega_B$ on the surrounding flow field. In this context, one may wish to evaluate $(\vu,p)$ at a collection of $n_{\text{eval}}$ points in $\Omega$, e.g., for visualization or postprocessing purposes. To this end, it suffices to execute \cref{algo:full} with \cref{eq:basiss} and subsequently evaluate the resulting surrogate at the desired target points. The overall computational complexity of this procedure can be derived as follows.

We first ignore the training procedure (the loop in \cref{algo:full}) and focus only on the evaluation task. If $n_{\text{iter}}$ iterations are carried out within \cref{algo:full}, the surrogate is the sum of $(2n_{\text{iter}}+1)$ terms: the surrogate flow field is
\begin{equation*}
    \widetilde\vu=\underbrace{\vu_0+\vv_{1/2}+\ldots+\vv_{n_{\text{iter}}-1/2}}_{\widetilde\vu_{\text{boundary}}}+\underbrace{\vv_0+\ldots+\vv_{n_{\text{iter}}-1}}_{\widetilde\vu_{\text{droplets}}}.
\end{equation*}
The surrogate pressure field has a similar expansion, which we ignore for simplicity in the following.

Note that, above, we have grouped together all half-step corrections with the baseline solution $\vu_0$ into $\widetilde\vu_{\text{boundary}}$, since all $(n_{\text{iter}}+1)$ addends relate to global Stokes solutions computed via $\text{S\_boundary}$, whose purpose is enforcing the boundary conditions on $\partial\Omega$, disregarding droplets. In a computational framework, we may assume that all such field have been computed using the same (e.g., FEM-based) discretization of \cref{eq:fem} (with different boundary conditions). By linearity, we may thus add all contributions together into a single ``vector of degrees of freedom'', so that evaluating $\widetilde\vu_{\text{boundary}}$ has $n_{\text{iter}}$-independent evaluation cost $c_{\text{boundary}}$.

By a similar token, we can exploit the similarity between all terms adding up to $\widetilde\vu_{\text{droplets}}$. Namely, we can express
\begin{equation*}
    \vv_j=\sum_{i=1}^M\sum_{k=1}^N\alpha_{i,j,k}\widehat\vv_k\circ\eta_i
\end{equation*}
where $M$ is the number of droplets, $\eta_i$ are the affine mappings defined in \cref{algo:full}, and $\alpha_{i,j,k}$ are the expansion coefficients $(\mathcal Q_N\vphi_{j-1/2,i})_k$, cf.~\cref{eq:basiss}. From here, we can compute
\begin{equation}\label{eq:surrdroplet}
    \widetilde\vu_{\text{droplets}}=\sum_{j=0}^{n_{\text{iter}}-1}\sum_{i=1}^M\sum_{k=1}^N\alpha_{i,j,k}\widehat\vv_k\circ\eta_i=\sum_{i=1}^M\underbrace{\left(\sum_{j=0}^{n_{\text{iter}}-1}\sum_{k=1}^N\alpha_{i,j,k}\widehat\vv_k\right)\circ\eta_i}_{\widetilde\vv_i:=}.
\end{equation}
Both indices $j$ and $k$ in the above expression can be collapsed, which leads to precomputing the $M$ fields $\widetilde\vv_i$, each corresponding to the total effect that a droplet has on the field. In this way, we can make the cost of evaluating $\widetilde\vu_{\text{droplets}}$ independent of $N$ and $n_{\text{iter}}$.

On the other hand, it is not computationally possible to achieve an accurate representation of $\widetilde\vu_{\text{droplets}}$ without incurring an $\mathcal{O}(M)$ scaling, e.g., by collapsing the ``droplet index'' $i$ in \cref{eq:surrdroplet}. Indeed, each droplet contributes to the total solution with a field that displays very localized features around the droplet itself, cf.~our numerical results below. An accurate numerical representation of the total field $\widetilde\vu_{\text{droplets}}$ must thus necessarily resolve such effects, resulting in a linear scaling in $M$.

In summary, the cost of a single surrogate evaluation scales as $\mathcal{O}(Mc_{\text{droplet}}+c_{\text{boundary}})$, with $c_{\text{droplet}}$ being the $n_{\text{iter}}$- and $N$-independent cost of evaluating each of the fields $\widetilde\vv_i$ in \cref{eq:surrdroplet} at a single point.

\begin{remark}\label{rem:femcost}
    In several settings, e.g., if many droplets are present ($M\gg1$) or if \cref{eq:fem} is discretized on a coarse mesh, the cost $c_{\text{boundary}}$ is negligible. In this regime, which we assume for simplicity from now on, the total cost of a single surrogate evaluation scales linearly with $M$.
\end{remark}

We now go back to the training procedure, consisting of an iterative correction loop. At the $j$-th iteration, the dominant costs arise from four sources:
\begin{enumerate}
    \item Assembling the flow field $\vphi_{j-1/2,i}$ for each droplet $i=1,\ldots,M$. Since this requires evaluating the current surrogate, cf.~the definition of $\vphi_{j-1/2,i}$, this has asymptotic cost $\mathcal{O}(M^2c_{\text{droplet}})$, cf.~\cref{rem:femcost}.
    \item\label{step:exp} Solving \cref{eq:singledroplet} to find the volume corrections $(\vv_{j,i},q_{j,i})$ for each droplet $i=1,\ldots,M$. Assuming that \cref{eq:basiss} is used and that the basis-expansion coefficients $\alpha_{i,j,k}$ are computed by interpolation, we can expect an $\mathcal{O}(M^2Nc_{\text{droplet}}+MN^3)$ cost, due to having to evaluate each $\vphi$ at $N$ points and solve $M$ size-$N$ interpolation problems. The complexity can be simplified to $\mathcal{O}(M^2Nc_{\text{droplet}})$ if $M\gg N$.
    \item Assembling and evaluating the boundary errors $\vgcorr{j+1/2}$ and $\vfcorr{j+1/2}$ at $\Gamma_D$ and $\Gamma_N$, respectively. Assuming that $n_{\text{BC}}$ boundary evaluations are needed in order for $(\vu_{j+1/2},q_{j+1/2})$ to be then computed through $\text{S\_boundary}$, this costs $\mathcal{O}(Mc_{\text{droplet}}n_{\text{BC}})$.
    \item Solving \cref{eq:fem} within $\text{S\_boundary}$ to find $(\vu_{j+1/2},q_{j+1/2})$. We denote the corresponding cost by $c_{\text{Stokes}}$.
\end{enumerate}

In the setting described in \cref{rem:femcost}, the second step has the largest complexity. Notably, the last step has especially low cost $c_{\text{Stokes}}$ if the matrix associated with the discretization of \cref{eq:fem} is pre-factorized as suggested in \cref{sec:full:full}. In such cases, the training complexity scales as $\mathcal{O}(M^2Nc_{\text{droplet}}n_{\text{iter}})$ overall.

\begin{remark}\label{rem:bie}
    More generally, the cost of the most expensive step (\ref{step:exp} above) depends on the specific strategy used to obtain $(\vv_{j,i},q_{j,i})$. For instance, computing such fields ``on the fly'' by boundary-integral equations, discretized using $N_{\textup{BIE}}$ degrees of freedom, costs approximately $\mathcal{O}(M^2N_{\textup{BIE}}c_{\textup{droplet}}+MN_{\textup{BIE}}^3)$.    
\end{remark}

Putting all together yields a complexity of
\[
    \mathcal{O}\left(M\big(MNn_{\text{iter}} + n_{\text{eval}}\big)c_{\text{droplet}}\right).
\]
In our experiments, we have observed that, in some cases, the cost of the evaluation phase may dominate, especially when $n_{\text{eval}}$ is large. This can occur, for example, when targeting high-resolution visualizations of the solution field. Concerning this, it is worth noting that surrogate evaluations are ``embarassingly parallelizable'', which may allow for substantial speedups in practice.

\begin{remark}\label{rem:dd}
    The quadratic scaling in $M$ is a symptom of the pairwise interactions between droplets during training. It might be possible to accelerate the training, reducing its complexity to $M\log M$, through spatially hierarchical strategies in the style of the fast multipole method (FMM) \cite{multipole}. To achieve this, one leverages the decay of the interaction forces (here, the the single-droplet solutions $\{\mathcal S(\widehat\vphi_k)\}_k$), cf.~also \cref{lem:singlewellposed}, in order to replace exact one-to-one interactions between distant droplets with approximate many-to-many interactions. Work on this idea is under way, although we include some preliminary scaling results in \cref{sec:num:scaling}.
\end{remark}

\subsection{Pipeline 2: time-domain simulation of moving droplets}\label{sec:time}

A second setting with practical relevance is that where the positions and shapes of the droplets $\Omega_B$ vary with time, since droplets are advected and deformed by the flow field solving \cref{eq:cutfem}. More rigorously, we consider the time-dependent problem \cref{eq:cutfemtime}.

To handle this problem efficiently, we propose a surrogate-based evolution strategy. The key idea is to decouple the geometric evolution from the fluid solve by leveraging the surrogate model to rapidly evaluate the local velocity field acting on each droplet. At each time step $t_n$ in a time grid\footnote{An adaptive time-stepping scheme based on the surrogate solution may of course be applied, e.g., where finer steps are taken if droplets are about to collide.} $0=t_0<t_1<\ldots<t_{N_{\text{time}}}$, the following procedure is carried out:
\begin{enumerate}
    \item\label{item:step} Given the current droplet configuration $\Gamma(t_n)$, use \cref{algo:full} to evaluate the surrogate model to obtain the approximate velocity field $\widetilde\vu(t_n)$.
    \item Place $n_{\text{adv}}$ ``markers'' on each droplet (e.g., uniformly spaced on its boundary) $\{\vx_{l,i}(t_n)\}_{l=1}^{n_{\text{adv}}}\subset\Gamma_i(t_n)$ for $i=1,\ldots,M$. Evaluate $\widetilde\vu(t_n)$ there.
    \item Update droplet positions and shapes by a forward-Euler step over all droplets' boundaries: \[\vx_{l,i}(t_{n+1})=\vx_{l,i}(t_n)+(t_{n+1}-t_n)\widetilde\vu(t_n)|_{\vx_{l,i}(t_n)}.\]
    \item Update $n\gets n+1$ and go to step \ref{item:step}.
\end{enumerate}
This approach allows the flow–interface coupling to be approximated without solving a full Stokes system at every time step.

It is crucial to observe that the above strategy will generally result in substantial deformations in the droplets' shapes, since the markers $\{\vx_{l,i}\}_{l=1,i=1}^{n_{\text{adv}},M}$ are advected independently. This ultimately may result in unstable behavior, requiring the use of relaxation strategies \cite{Biros19}. Moreover, as our methods have been so far developed for circular droplets, we cannot yet handle general shape deformations. We refer to \cref{sec:extensions} for discussion on that.

To recover tractability, we suggest a simplified advection strategy, which shifts each droplet by a rigid translation by averaging the corresponding markers' velocities:
\begin{equation}\label{eq:timestep_markers}
    \vx_{l,i}(t_{n+1})=\vx_{l,i}(t_n)+(t_{n+1}-t_n)\frac1{n_{\text{adv}}}\sum_{l'=1}^{n_{\text{adv}}}\widetilde\vu(t_n)|_{\vx_{l',i}(t)}.
\end{equation}
This formula preserves droplet shapes at the cost of accuracy, especially when droplets deform substantially as a result of interactions with nearby droplets or boundaries. Moreover, it may cause droplets to unphysically overlap, rather than deform or merge, in case of collisions. The resulting overlap is not interpreted as a physical prediction. Within the non-contacting regime, we believe that the resulting approximation errors may be sufficiently small to justify using our approach to recover tractability for otherwise unfeasible problems involving large numbers of droplets. The time-dependent simulations considered here are restricted to configurations in which the droplets remain distinct and separated by a positive interfacial gap throughout the simulated time interval. The underlying time-dependent Stokes model allows for droplet deformation, but the present strategy restricts the droplets to rigid translations in order to preserve the circular shapes required by the surrogate model. The method is thus intended to capture interactions between closely spaced droplets, including configurations with small gaps, but not the deformation or coalescence processes that may occur upon contact. No collision or coalescence model is included in the present numerical formulation; non-overlap is instead an assumption on the configurations considered. Ad hoc collision handling could alternatively be introducted to prevent overlaps, for example by ``nudging'' droplets apart. However, our preliminary testing indicates that such a treatment does not systematically improve the accuracy of the surrogate model in time-dependent simulations. For this reason we do not introduce any collision-prevention mechanism in the present time-domain surrogate model. Importantly, although the numerical method can continue to function even when droplets overlap, as demonstrated in \cref{sec:numerics:conv}, such configurations are non-physical and lie outside the intended physical range of applicability.

The accuracy of \cref{eq:timestep_markers} is tested numerically in \cref{sec:numerics}. The computational cost of the resulting time-domain simulation scales as $\mathcal{O}(M^2(Nn_{\text{iter}} + n_{\text{adv}})c_{\text{droplet}}N_{\text{time}})$.

The complexity of the two pipelines is summarized in \cref{tab:scaling}, where we also include the improved scalings obtained through the FMM-based acceleration idea described in \cref{rem:dd}. The table entries are numerically validated in \cref{sec:num:scaling}.

\begin{table}[ht]
    \centering
    \begin{tabular}{c|c|c}
        \textbf{Offline phase} & \multicolumn{2}{c}{Complexity}\\
        \hline
        Compute $\{\mathcal S(\widehat\vphi_k)\}_{k=1}^N$ & \multicolumn{2}{c}{$Nc_{\text{BIE}}\simeq NN_{\text{BIE}}^3$}\\[2pt]
        \hline
        \hline
        \textbf{Online phase} & Complexity without FMM & Complexity with FMM\\
        \hline
        Time-independent training & $M^2Nn_{\text{iter}}$ & $M\log MNn_{\text{iter}}$\\
        Time-independent evaluation & $Mn_{\text{eval}}$ & $\log Mn_{\text{eval}}$\\
        \hline
        Full time-dependent simulation & $M^2(Nn_{\text{iter}} + n_{\text{adv}})N_{\text{time}}$ & $M\log M(Nn_{\text{iter}} + n_{\text{adv}})N_{\text{time}}$\\
    \end{tabular}
    \caption{Expected scalings of \cref{algo:offline} in the time-independent and -dependent frameworks, assuming for simplicity that $c_{\text{droplet}}=\mathcal O(1)=c_{\text{boundary}}$.}
    \label{tab:scaling}
\end{table}

\section{Extension to general shapes}\label{sec:extensions}
As discussed in \cref{sec:basis}, an efficient solve of \cref{eq:singledroplet} within \cref{algo:full} assumes geometric similarity between each droplet and the reference configuration (the unit disk). This assumption allows a speed-up in each call to $\text{S\_droplet}$ from approximately $\mathcal{O}(N_{\text{BIE}}^3)$, with $N_{\text{BIE}}$ being the number of degrees of freedom used to discretize \cref{eq:singledroplet} by boundary-integral equations, cf.~\cref{rem:bie}, to $\mathcal{O}(N^3)$, with $N$ the number of single-droplet solutions used in \cref{eq:basiss}.

However, droplets may generally differ in shape. In particular, large deformations  may occur during time-dependent simulations of  \cref{eq:cutfemtime}, especially when droplets approach one another or domain boundaries. 
Such interactions can generate strong lubrication forces and may even lead to coalescence. The strategy described so far cannot handle large shape deformations without sacrificing either accuracy, by assuming the droplets are undeformable and neglecting the resulting changes in the flow field, or efficiency, by using \cref{algo:full} without the subspace-based acceleration.

In our numerical tests below, we choose the former, simpler approach, neglecting droplet deformations at the cost of accuracy. In this setting, each time step induces an easy-to-handle, rigid-body motion of each droplet. This idea is described in more detail in \cref{sec:time} and tested in \cref{sec:timetest}. At the same time, we acknowledge that this approach may yield unsatisfactory results in applications involving large droplet deformations. As such, ongoing work is being carried out to approximate the effect of deformations through \emph{shape parametrization}, in the following sense.

Depending on the application, it might be possible to predict in which ways the droplets will deform, allowing one to design a \emph{parametrized} family of shapes that the droplets are ``allowed'' to take \cite{minakowska_finite_2021,von_wahl_using_2021}. In this setting, the collection of single-droplet solutions corresponding to such shapes becomes parameter-dependent too. Cheaply approximating the single-droplet solutions induced by an arbitrary parameter choice (i.e., a given droplet shape) may be feasible through interpolation or machine-learning techniques \cite{Biros19,minakowska_finite_2021} starting from a discrete database of single-droplet solutions, obtained by ``sampling'' the parametrized family of shapes. As mentioned above, this extension is subject of ongoing work.

As a final note, we also mention that the assumption of circular droplets allows drastic simplifications in the handling of surface tension, cf.~\cref{sec:surrogate}. If droplets have general shapes, no analytic correction exists to account for surface tension. As such, surface tension remains part of the interface force balance that is corrected through iteration: a nonzero $\widetilde\gamma$ must be included in \cref{eq:singledroplet}. Iterated correction of the surface tension term can be handled within the call to \textup{S\_droplet}, along with the other volume and interface correction terms. See also the discussion following \cref{algo:offline} concerning handling this case within the offline-online framework.

\section{Numerical examples}\label{sec:numerics}

We perform a suite of numerical tests to validate the accuracy and efficiency of the proposed method for multiphase Stokes problems in $d=2$ spatial dimensions. We consider a viscosity contrast of  $100$, with piecewise-constant viscosity $\mu=\mu_F=1$ outside the droplets and $\mu=100$ inside them. We set the volume forcing $\vforce=\vzero$ and the surface tension coefficient $\gamma=10$. As a high-fidelity reference, we employ CutFEM, which accurately discretizes the Stokes problem on domains with complex outer-boundary geometries while representing large numbers of immersed droplet interfaces without requiring the mesh to conform to them. Unless otherwise specified, we employ the following surrogate setup:
\begin{itemize}
    \item We use \cref{algo:full}, stopping the fixed-point iterations, indexed by $j$, according to the criterion
    \begin{equation}\label{eq:fixedhalflinf}
        \norm{\vv_j}_{L^\infty(\Gamma_D)} + \norm{\uT(\vv_j,q_j;\mu)}_{L^\infty(\Gamma_N)}\leq\textup{tol}=10^{-3},
    \end{equation}
    cf.~\cref{sec:termination}. This choice of $\textup{tol}$ is validated in our first test.
    \item We use \cref{eq:basiss} to more efficiently solve the single-droplet auxiliary problem \cref{eq:singledroplet}. Namely, the flow field $\vphi$ is approximated through a polynomial of degree $D=2$, which results in $N=6$ single-droplet solutions, cf.~\cref{sec:basis,app:poly}. The subspace projection $\vphi\mapsto\vphi_N$ is carried out via least-squares interpolation using $8D$ uniformly distributed points on $\Omega_B^i$, which is intended to mimic $L^2(\Omega_B^i)$-projection by quadrature. This choice of $D$ is validated in our first test.
    \item The lowest-fidelity solution $(\vu_0,p_0)$ is computed by the finite-element method (FEM) on a relatively coarse grid. The same method is used also to solve \cref{eq:fem} through $\text{S\_boundary}$ within \cref{algo:full}. For the sake of efficiency, a sparse LU factorization of the FEM matrix is precomputed and stored to disk.
    \item To test accuracy, we compute reference solutions using CutFEM~\cite{acta25}. The CutFEM discretization is described in \cref{app:cutfem}. While the overall framework follows \cite{FraZah19}, we employ a nonsymmetric Nitsche formulation and the lowest-order Taylor--Hood finite-element pair (P2--P1). The ghost-penalty stabilization parameters for both the velocity and pressure are set to one. For the numerical representation of the interfaces, we use the level-set method \cite{OshFed01,Sethian01}, where the scalar function $\phi=\phi(\vx)$ is used to characterize the droplet region $\Omega_B:=\{\vx\in\Omega:\phi(\vx)>0\}$. The approximation error is computed for the velocity field in the $L^2(\Omega)$ and $L^\infty(\Omega)$ norms.
\end{itemize}

We run \cref{algo:full} in Python 3.12 on a workstation with an Intel\textsuperscript{\textregistered} Core\textsuperscript{\texttrademark} Ultra 7 processor 265 and $64$GB of RAM. On the other hand, due to the higher memory requirements, especially in our scaling tests in \cref{sec:num:scaling}, the CutFEM code that we use to compute reference solutions is run via MPI on 4 CPU nodes of the Dardel cluster at KTH, each with an AMD EPYC 7742 CPU and $256$GB of RAM. FEM and CutFEM are implemented in C++ as part of the open-source CutFEM library \cite{github}.

\begin{center}
\begin{framed}
Reproducibility code is available at \url{https://github.com/CutFEM/hierarchical2DDropletStokes}.
\end{framed}
\end{center}

\subsection{Convergence test}\label{sec:numerics:conv}

We first study the accuracy of our method by looking at the approximation error in the flow field $\vv^{(n_{\textup{iter}})}=\vu_{n_{\textup{iter}}}-\vu$. We consider a 2D unit square domain modeling a simple right-angle pipe junction: homogeneous Dirichlet conditions are imposed on the bottom and right boundaries, whereas two parabolic flow profiles (one in-going, one out-going) are imposed on the top and left boundaries through Dirichlet condition. The CutFEM reference solution is computed on a uniform mesh of size $h=1/1500$.

\begin{figure}
    \centering
    \includegraphics{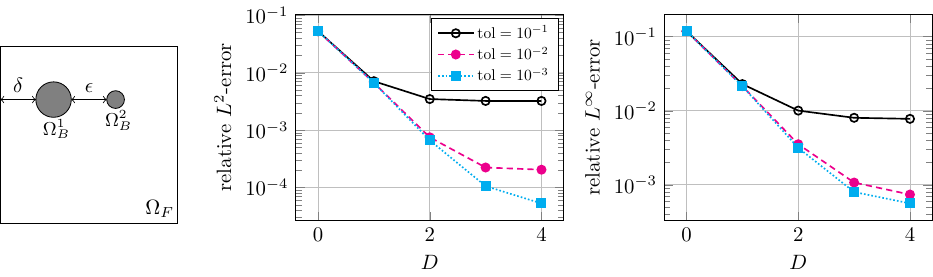}
    \caption{Left: numerical setup with $M=2$ droplets at distance $\delta$ from the boundary and $\varepsilon$ from each other. Middle and right: convergence of surrogate error with respect to truncation degree $D$ for $\delta=\epsilon=0.2$, measured in the $L^2(\Omega)$ and $L^\infty(\Omega)$ norms, respectively.}
    \label{fig:conv}
\end{figure}

We consider two droplets ($M=2$) with radii $r_1=0.1$ and $r_2=0.05$, centered at $\vx_1=(0.3,0.7)$ and $\vx_2=(0.65,0.7)$. We first study how the approximation error changes when varying (i) the approximation degree $D$ used to approximate \cref{eq:singledroplet} with \cref{eq:basiss} and (ii) the tolerance $\textup{tol}$ set on the boundary residual norm \cref{eq:fixedhalflinf}, used to stop the fixed-point iterations in \cref{algo:full}. Recall that the numbers $N$ of single-droplet solutions needed for degree $D$ are $N=0$ for $D=0$ and $N=4D-2$ for $D>0$.

The number of fixed-point iterations increases as $\textup{tol}$ decreases: tolerances $10^{-1}$, $10^{-2}$, and $10^{-3}$ require $1$, $3$, and $5$ iterations, respectively. The number of fixed-point iterations may generally depend on $D$ although no variations were observed here. The relative approximation errors are shown in \cref{fig:conv} (middle and right). First, we note that the approximation error using just lowest-degree single-droplet solutions ($D=1$) is already about one order of magnitude lower than that achieved by the lowest-fidelity solution $(\vu_0,p_0)$, which corresponds to $D=0$.

An early termination of the fixed-point iterations due to a too large $\textup{tol}$ leads to saturation with respect to the polynomial degree $D$. Indeed, despite the more accurate approximation of the local volume residual on the droplet, cf.~\cref{sec:basis}, large residuals are left when the fixed-point iterations end. Focusing now on the smallest tested value of $\textup{tol}$, we observe exponential convergence with respect to $D$ for both $L^2$ and $L^\infty$ errors. This is expected since we are performing a spectral (truncated-Taylor) approximation of a smooth function, namely, the flow field $\vphi$ that is given as input to the linear operator $\mathcal S$, cf.~\cref{sec:basis}.

As a test to investigate the convergence behavior of the fixed-point iterations with respect to the geometric configuration of the droplets, we fix $D$ and place the droplets at different distances from the domain boundaries and from each other. See \cref{fig:conv} (left). More specifically, given $\delta$ and $\epsilon$, we place the droplets' centers at $\vx_1=(\delta+0.1,0.7)$ and $\vx_2=(\delta+\epsilon+0.25,0.7)$, keeping the radii fixed. This places the first droplet at distance $\delta$ from the left boundary and the two droplets at distance $\epsilon$ from each other, cf.~\cref{fig:conv} (left).

\begin{figure}
    \centering
    \includegraphics{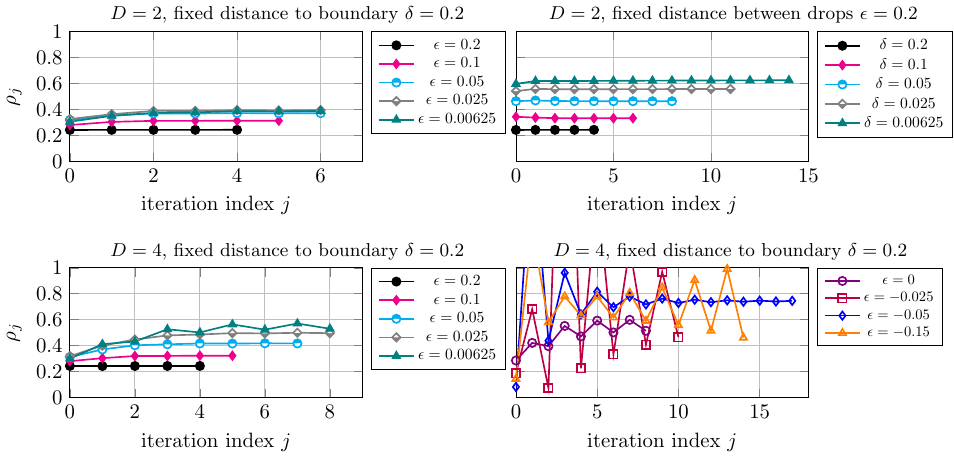}
    \caption{Contraction rate of fixed-point iterations for the setup in \cref{fig:conv} (left) with $\textup{tol}=10^{-4}$. The polynomial degree $D$ is reported above each plot. The bottom-right plot corresponds to non-physical configurations with overlapping droplets ($\varepsilon\leq0$).}
    \label{fig:iter}
\end{figure}

We choose a small $\textup{tol}=10^{-4}$, which results in more iterations than above, in order to allow the convergence behavior to fully develop, leading to a more clear comparison of different values of $\delta$ and $\epsilon$. We run \cref{algo:full} to build a surrogate model for each configuration, keeping track of the \emph{contraction rate}
\begin{equation*}
    \rho_j=\frac{\norm{\vv_{j+1}}_{L^\infty(\Gamma_D)} + \norm{\uT(\vv_{j+1},q_{j+1};\mu)}_{L^\infty(\Gamma_N)}}{\norm{\vv_j}_{L^\infty(\Gamma_D)} + \norm{\uT(\vv_j,q_j;\mu)}_{L^\infty(\Gamma_N)}}.
\end{equation*}
This provides us with information on the contractive properties of the fixed-point operator and, by proxy, on the convergence properties of the fixed-point loop: larger values of $\rho_j$ indicate slower convergence, with $\rho_j\geq1$ being a symptom of no convergence at all. The contraction rate is plotted in \cref{fig:iter} for four different configurations and varying values of $\delta$ and $\epsilon$. Convergence was reached in all cases. We can make the following remarks:
\begin{enumerate}
    \item The top two plots allow us to compare the effects of the distance $\delta$ between droplets and boundaries versus the distance $\epsilon$ between droplets only, since one of the two distances is kept fixed while the other varies. Decreasing either $\delta$ or $\epsilon$ increases the contraction rate, causing the need for more iterations, cf.~\cref{thm:convergenceD,rem:traditional}. However, a small distance $\delta$ to the domain boundary seems to have a stronger effect on convergence than a small distance $\epsilon$ between droplets, suggesting that ``hard'' boundaries hinder our method more than ``soft'' interfaces.
    \item The left two plots allow us to compare the effects of the truncation degree $D$ on convergence. We observe that nearby droplets seem to slow convergence down slightly more for larger degrees $D$. This is likely related to conditioning issues in iterating over the coefficients of higher-degree single-droplet solutions. We also observe slight oscillations developing in $\rho_j$ for $D=4$ and $\epsilon=0.00625$. These arise as symptoms of the strong two-way interactions between close droplets, which fixed-point iterations must approximate whenever $M>1$, cf.~\cref{sec:full:full,thm:convergenceD}. Despite the oscillations, the algorithm still behaves in a tame way, converging at a fairly rapid pace.
    \item In the bottom right plot, we stress-test our method for $\epsilon\leq0$, which corresponds to overlapping droplets. This is non-physical since the corresponding interface conditions in \cref{eq:cutfem} are incompatible. For small overlaps, the results are not so different from those in the bottom-left plot. On the other hand, more substantial overlaps lead to badly behaved iterations, with wild oscillations in the contraction rates, which in some cases exceed unity, a symptom of potential divergence. This also empirically confirms the need for a positive gap in order to obtain \emph{a priori} contractivity bounds, cf.~\cref{lem:singlewellposed}. This being said, our algorithm eventually reaches convergence in all tests, even when droplets are concentric ($\epsilon=-0.15$), although more iterations are needed as a result of larger contraction rates. This provides empirical evidence of the robustness of our iterative approach, which might be of relevance in overlapping or almost-overlapping cases, where standard techniques such as CutFEM or boundary-integral methods may fail or require substantially increased computational effort.
\end{enumerate}

\subsection{Time-domain simulation}\label{sec:timetest}

\begin{figure}
    \centering
    \includegraphics{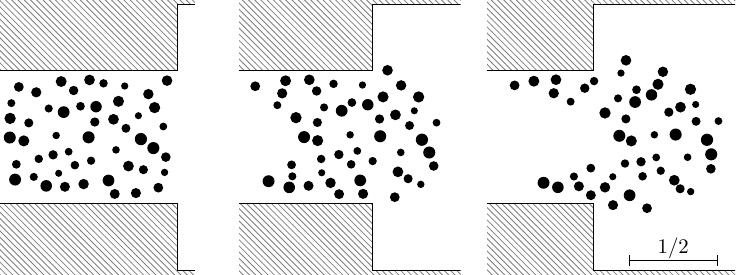}
    \caption{Snapshots of $M=50$ droplets at $t=0$ (left), $t=2$ (center), and $t=4$ (right).}
    \label{fig:snapshots}
\end{figure}

As a second, more realistic test, we consider a time-domain problem of the form \cref{eq:cutfemtime}. The problem concerns the time-evolution of $M=50$ droplets in a 2D rectangular pipe that suddenly widens, cf.~\cref{fig:snapshots}. The flow is driven by a simple parabolic inflow profile, whereas no-slip conditions are imposed on the pipe boundaries. The initial positions and radii of the droplets are randomly chosen. We fix a timestep $\Delta t=0.005$ and a number $N_{\text{time}}=800$ of time-steps, reaching the time horizon $t=t_{N_{\text{time}}}=4$.

Our ultimate target is approximating the $M$ trajectories of each droplet's center of mass. We compare four methods for approximating these:
\begin{itemize}
    \item CutFEM on a relatively fine mesh generated using Gmsh. The minimum and maximum mesh edge lengths are $\sim0.00114$ and $\sim0.0148$, respectively, resulting in $\sim4\cdot 10^6$ degrees of freedom. Since the droplets locations and shapes evolve in time, a time-dependent level-set function $\phi=\phi(t,\vx)$ is used. Given the level-set function at time $t$ and the fluid velocity $\vu(t,\vx)$ (computed by solving the Stokes equation with CutFEM \cref{app:cutfem}), the level-set function at time $(t+\Delta t)$ is found by evolving the interface $\Gamma(t)$ through a discretized version \cite{FraZah19} of the advection equation $\partial_t \phi + \vu \cdot \nabla{\phi} = 0$, valid in the whole $\Omega$ and with initial condition given from the level-set function at the previous time step. In order to avoid instabilities, the level-set function is reinitialized every $5$ time steps as described in \cite{SuFa99}. We treat the resulting trajectories from this scheme as ``reference'' in our tests. Each time-step takes about $2$ minutes, with $\sim25$ seconds spent on CutFEM assembly, $\sim75$ seconds spent on solving the discrete Stokes problem, and $\sim20$ seconds spent to assemble and solve the level-set-advection equation. Level-set reinitialization takes $\sim30$ additional seconds every $5$ time steps.
    \item Lowest-fidelity Stokes flow $\vu_0$. At each time-step, the droplets are advected as passive tracers, taking the mean of $\vu_0$ at $n_{\text{adv}}=100$ markers uniformly distributed on each droplet's boundary, cf.~\cref{eq:timestep_markers}. Since field $\vu_0$ is constant in time, this strategy is extremely cheap, taking only $\sim0.4$ milliseconds per time-step. This time does not include the one-time overhead of computing $\vu_0$, which, since no droplets are present in \cref{eq:fem}, can be done on a coarse mesh at extremely low cost ($\sim4$ seconds on our workstation).
    \item Our surrogate with the framework from \cref{sec:time}. We use the same advection strategy as above, although the field approximation resulting from \cref{algo:full} replaces the time-independent Stokes flow $\vu_0$. At each time-step, our method performs $5$ fixed-point iterations (automatically determined based on $\textup{tol}$) given the current droplet configuration, resulting in an average compute time of $\sim1.5$ seconds per time-step. About $80\%$ of this time is spent solving Stokes problems of the form \cref{eq:fem} by FEM on a coarse mesh (using triangular solvers on LU-pre-factorized FEM matrices).
    \item CutFEM without droplet deformations. We run the CutFEM code on the same mesh, but assuming that droplets do not deform and thus remain circular throughout the simulation. This assumption, consistently with our surrogate, slightly reduces the computational cost of advecting the droplets since the level set can be advected through rigid-body motion. Each time-step takes about $100$ seconds, without the need for any reinitialization.
\end{itemize}

\begin{figure}
    \centering
    \includegraphics{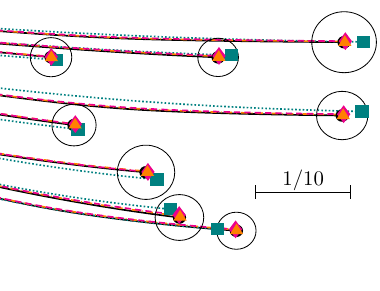}~\hspace{2mm}~\includegraphics{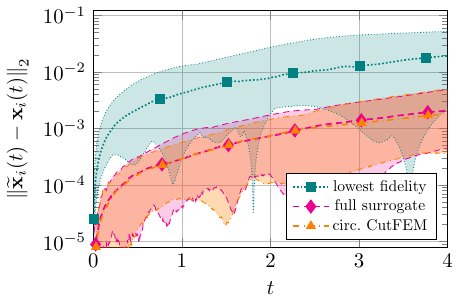}
    \caption{Left: Zoom of droplet trajectories near the bottom-right-most droplet. We show reference solution (black circles, full lines), lowest-fidelity surrogate (teal squares, dotted lines), full surrogate (magenta diamonds, dashed lines), and no-deformation CutFEM solution (orange triangles, dash-dotted lines). For readability, markers are shown only at the final time $t=4$. Thin black lines denote boundaries of reference droplets at the final time. Right: error in lowest-fidelity, full surrogate, and no-deformation CutFEM trajectories over time. Each envelope shows the spread of the error, with lines marking the minimum, median, and maximum error over all droplets.}
    \label{fig:trajectory}
\end{figure}

In \cref{fig:trajectory} (left) we display some of the droplets' trajectories obtained using the four methods. The surrogate trajectories are rather close to the reference ones, especially when compared to the lowest-fidelity results. This observation is confirmed quantitatively in \cref{fig:trajectory} (right), where we show the trajectory approximation error. We can see that the full surrogate lowers the error by about an order of magnitude, compared to the lowest-fidelity approximation. Notably, the surrogate achieves a $\times80$ speed-up with respect to the reference code while attaining a fair accuracy. The lowest-fidelity ``passive-tracer'' surrogate is also remarkable, since, despite a one-order-of-magnitude-larger error, it achieves a $\times3\cdot10^5$ speed-up (excluding the single Stokes solve).

We look now at the non-deformable CutFEM results, whose computational cost is about $60$ times larger than our full surrogate's. As shown in \cref{fig:trajectory} (right), the method's accuracy is extremely close to our surrogate's. This proves that the main source of error in our surrogate strategy is the circularity assumption, not the iterations or the subspace approximation of S\_droplet. As such, we have good hopes for the accuracy achievable by our approach, once deformations are incorporated in the pipeline through future work.

As a final comparison, we also tried to run our CutFEM codes, both with and without droplet deformation, on a coarsened spatial mesh with about $20$ times less degrees of freedom. The associated runtime per iteration was about $9$ seconds in both cases, making these approaches comparable to our surrogate in terms of cost. However, the decreased computational cost comes with a decreased accuracy.

In particular, the coarsened results with deformation were wildly inaccurate due to droplets unphysically merging into ``peanut'' shapes due to the low mesh resolution (interfaces can merge if they intersect a single mesh element). Once merged, droplet pairs are pulled together by surface tension, thus increasing trajectory errors at further times. On the other hand, the errors in the non-deformable, coarsened test were only roughly twice as large as in the non-deformable, non-coarsened simulation. This seems to suggest that coarsening might be a beneficial strategy for time savings as long as droplets are handled as rigid bodies. However, we must mention that, in order to achieve this level of accuracy, we had to manually ``turn off'' surface tension to prevent the above-mentioned persistent errors due to merging, a strategy that is physically justified only if droplets are circular, cf.~\cref{sec:surrogate}.

\subsection{Scaling test}\label{sec:num:scaling}

As a last numerical experiment, we investigate the scaling of our method, in terms of both accuracy and efficiency, as the number of droplets $M$ increases. To this aim, we consider the static problem \cref{eq:cutfem} over a rectangular domain of width $3$ and height $1$, with $10^2\leq M\leq10^4$ droplets at (quasi-uniform) random locations. The droplets' radii scale as $r_i=\mathcal O(\sqrt M)$ to keep the ``volume fraction'' $|\Omega_B|/|\Omega|$ constant at about $20\%$. For simplicity, the flow is driven by parabolic inflow and outflow profiles at the left and right boundaries, respectively. See \cref{fig:scaling_qual}.

For each $M$, the reference flow field is computed via CutFEM on a uniform mesh with size $h\approx0.0535/\sqrt{M}$, yielding $h=3.78\times10^{-3}$ for $M=200$ and $h=7.57\times10^{-4}$ for $M=5000$, for example. The scaling $h=O(1/\sqrt M)$ is consistent with the droplet-radius scaling (used to maintain a fixed droplet volume fraction) and keeps approximately a fixed number of mesh cells across each droplet as $M$ increases, ensuring that the smallest droplet-droplet and droplet-boundary gaps remain resolved. This results in $\mathcal O(M)$ degrees of freedom, which implies that the computational cost of the CutFEM solver scales approximately as $\mathcal O(M)$, assuming an ideal cost of the linear-system solver.

\begin{figure}[tb!]
    \centering
    \includegraphics{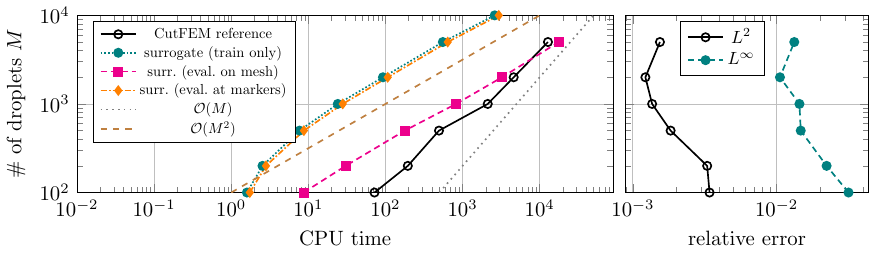}\vspace{-1mm}
    \caption{Left: timing of surrogate training, compared to CutFEM assembly and solve. Total runtimes of two surrogate pipelines: surrogate field evaluation at all CutFEM degrees of freedom (\cref{sec:field}) and at marker points (\cref{sec:time}). Right: global relative approximation error in the velocity field.}
    \label{fig:scaling}
\end{figure}

The reference results are compared to the field produced by \cref{algo:full}, which can be run on our workstation due to the much lower RAM requirements (e.g., $\sim40$ GB for $M=5000$, which can be drastically lowered at the cost of efficiency by reducing caching). See the beginning of \cref{sec:numerics} for full details on the hardware used in our tests. For a fair comparison, after the surrogate is ``trained'' through \cref{algo:full}, we evaluate the surrogate flow field on a fine grid with $n_{\text{eval}}=\mathcal O(M)$ points, namely, the locations of the degrees of freedom in CutFEM. Incidentally, note that this large amount of evaluations is needed to compare surrogate and reference flows in the context of an error analysis, although it might be excessive in practical applications.

The results are summarized in \cref{fig:scaling}. Therein, we can observe that, given the large number of evaluation points $n_{\text{eval}}$, the surrogate evaluation takes almost $100$ times longer than the surrogate training. In agreement with the analysis in \cref{sec:timing}, the surrogate evaluation at $\mathcal O(M)$ points takes $\mathcal O(M^2)$ time, with each point-evaluation taking $\mathcal O(M)$ time.

We observe the same quadratic scaling also in the surrogate training for large $M$, which suggests that (i) the cost of solving Stokes problems ($\text{S\_boundary}$ in \cref{algo:full}) becomes negligible for large $M$, as suggested in \cref{sec:field}, and (ii) the number of fixed-point iterations $n_{\text{iter}}$ remains roughly constant as $M$ increases. In our experiments, we actually observed a decrease in $n_{\text{iter}}$ as $M$ increases, cf.~\cref{fig:scaling_dense} (right). This indicates that, at least in this setting with a fixed volume fraction, interaction effects between smaller droplets are easier to approximate through our method's fixed-point iterations. As partial evidence supporting this, we note that the approximation error generally decreases as $M$ increases even though $D$ is kept constant, cf.~\cref{fig:scaling} (right).

Including also the surrogate evaluation, our algorithm is more efficient than the reference CutFEM for $M<5000$. At $M=5000$ the surrogate's quadratic scaling finally catches up with the almost linear cost of CutFEM. On the other hand, if one excludes the surrogate evaluation, our method is several orders of magnitude more efficient than the reference technique for all tested values of $M$. This means that our surrogate-modeling strategy is competitive if a sufficiently small(er) number of evaluations is required. For instance, this happens when the target problem is a transient time evolution as in \cref{sec:time}, which requires evaluating the surrogate field at $100M$ total marker points (when $n_{\text{adv}}=100$), rather than at the $\sim6700M$ CutFEM degrees of freedom. In both cases, $n_{\text{eval}}$ scales linearly in $M$, but with vastly different proportionality constants. The corresponding timings are also included in \cref{fig:scaling} (left) and \cref{fig:scaling_dense} (center).

\subsubsection{Higher volume fraction}
In order to show the effect of higher volume fraction on our method's efficiency, we repeat our experiment with larger droplets, corresponding to $|\Omega_B|/|\Omega|\simeq 50\%$. This results in stronger interactions between droplets, making more fixed-point iterations necessary to attain the same tolerance $\textup{tol}$. This is shown in \cref{fig:scaling_dense} (right), where we notice that $n_{\textup{iter}}$ seems to stabilize at a (volume-fraction-dependent) value as $M$ increases. This value represents a sort of ``homogenized'' (for $M\to\infty$ with fixed volume fraction) fixed-point complexity for the corresponding volume fraction. Since a larger volume fraction generally yields smaller gaps between droplets as well as between droplets and domain boundaries, this result agrees with those presented in terms of contraction rate $\rho_j$ in \cref{fig:iter}.

\begin{figure}[tb!]
    \centering
    \includegraphics{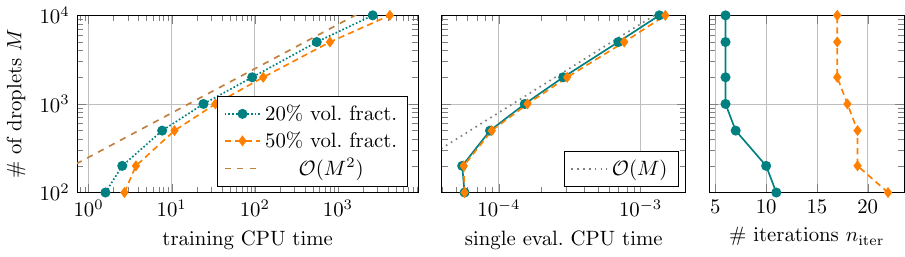}%
    \vspace{-1mm}
    \caption{Comparison of low and high volume fraction. Left: timing of surrogate training. Center: timing of each surrogate point-evaluation. Right: number of fixed-point iterations needed to attain $\textup{tol}=10^{-3}$.}
    \label{fig:scaling_dense}
\end{figure}

Despite this difference in number of iterations, the surrogate training costs increases only slightly with the volume fraction, cf.~\cref{fig:scaling_dense} (left). This is because the most expensive steps in \cref{algo:full}, namely, the $\mathcal O(M^2N)$ evaluations of the single-droplet solutions $\{\mathcal S(\widehat\vphi_k)\}_{k=1}^N$ in \cref{eq:basiss}, are cached at the first fixed-point iteration for efficiency, making the computational cost almost independent of the number of iterations. The surrogate evaluation costs are basically the same for both volume fractions, cf.~\cref{fig:scaling_dense} (center).

This test provides further evidence of how our method's computational cost does not drastically blow up when the droplet interactions strengthen as a result of small gaps between droplets, e.g., those that might naturally arise in a time-domain simulation involving \cref{eq:cutfemtime}. On the other hand, more standard discretization methods, e.g., based on CutFEM or boundary-integral equations, would face great difficulties in dealing with such small droplet gaps. While some evidence on this was provided in \cref{sec:timetest}, we believe it would be of interest to carry out further, more detailed comparisons of our algorithm with other state-of-the-art methods.

\subsubsection{Fast-multipole-method-based acceleration}
As a final experiment, we test also an FMM-accelerated version of our method, where exact droplet-droplet and droplet-boundary interactions are replaced by inexact approximation whenever a large distance is present between a droplet and the location where its generated correction field $\vv_{j,i}$ is to be evaluated. We perform this test as a way to provide preliminary evidence in support of \cref{rem:dd}, showing that the $\mathcal O(M^2)$ and $\mathcal O(M)$ scalings of surrogate training and surrogate evaluation may be reduced to $\mathcal O(M\log M)$ and $\mathcal O(\log M)$, respectively.

\begin{figure}[tb!]
    \centering
    \includegraphics{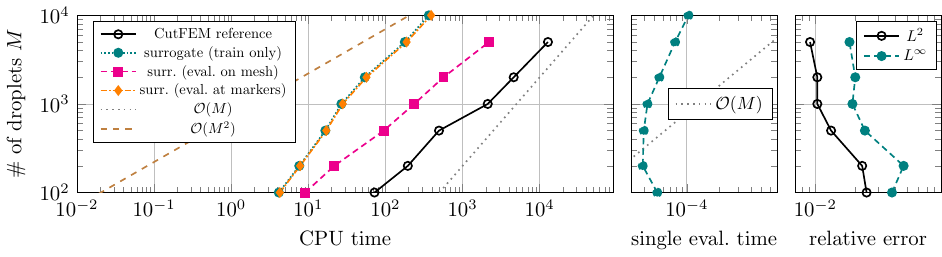}%
    \vspace{-1mm}
    \caption{Test with FMM acceleration. Left: timing of surrogate training, compared to CutFEM assembly and solve. Total runtimes of two surrogate pipelines: surrogate field evaluation at all CutFEM degrees of freedom (\cref{sec:field}) and at marker points (\cref{sec:time}). Center: timing of each surrogate point-evaluation. Right: global relative approximation error in the velocity field.}
    \label{fig:scaling_dd}
\end{figure}

Our results with FMM acceleration in the $20\%$-volume-fraction case are displayed in \cref{fig:scaling_dd}, which should be compared to \cref{fig:scaling} and \cref{fig:scaling_dense} (center). For few droplets, the training times are slightly larger with FMM, a symptom of the FMM-related overhead (namely, the initialization of the geometrical hierarchy). However, for larger values of $M$, the training costs eventually become lower than those without FMM even with our unoptimized code, with a speed-up factor of $\sim7$ for $M=10^4$, with respect to the surrogate without FMM. The measured complexity of the surrogate-training phase with FMM is clearly sub-quadratic-in-$M$ and compatible with the expected $\mathcal O(M\log M)$ rate.

The efficiency benefit of FMM can be clearly seen also by looking at the surrogate-evaluation times in \cref{fig:scaling_dd} (center), which increase at a sub-linear-in-$M$ rate, to be compared to the linear rate without FMM. This also makes the surrogate evaluation on the CutFEM mesh much cheaper, resulting in substantial computational speed-up with respect to the CutFEM reference, as shown in \cref{fig:scaling_dd} (left), whereas the surrogate without FMM performs worse than CutFEM for $M=5000$, cf.~\cref{fig:scaling} (left).

As shown in \cref{fig:scaling_dd} (right), the surrogate errors are below $5\%$ in both $L^2$ and $L^\infty$ norms. However, incorporating FMM increases the approximation errors by a factor $5\div10$ for the $L^2$ error and $1\div2$ for the $L^\infty$ error, with respect to the surrogate without FMM. This is a consequence of the hierarchical field approximation, whose accuracy is high near the droplets (hence the small increase in $L^\infty$ error, as the largest errors are found near droplet boundaries $\Gamma$) but deteriorates further away (hence the larger increase in the $L^2$ error). The rather substantial increase in the $L^2$ error is most likely due to sub-optimal hyperparameter selection in our preliminary implementation of FMM. In the future, we aim to improve the performance and accuracy of this accelerated method, with the objective of achieving similar speed-ups with a lower error, more comparable to that achieved without FMM acceleration.

\section{Conclusions} \label{sec:conc}
We have presented a multi-fidelity surrogate modeling strategy for Stokes flows with large numbers of suspended droplets. The approach alternates between a background Stokes solver and precomputed single-droplet disturbance fields, using an iterative error-PDE correction to systematically capture droplet–flow, droplet–boundary, and droplet–droplet interactions. For geometrically similar droplets, an offline–online decomposition enables efficient reuse of single-droplet solutions. Numerical experiments demonstrate that the method maintains good accuracy across a wide range of configurations, including systems with up to $10^4$ droplets. Even within a non-optimized implementation, the surrogate yields substantial speed-ups over fully resolved CutFEM simulations, demonstrating its potential to bring complex multiphase flow simulations within the reach of standard workstations, rather than requiring supercomputers.

The present framework naturally leads to several directions for future research. First, extending the framework to droplets with variable shapes appears to be an important target to recover accuracy in settings with large droplet deformations, as outlined in \cref{sec:extensions}. Second, our preliminary tests demonstrate the potential of incorporating an FMM-style acceleration into our framework, as outlined in \cref{rem:dd}, although further refinements in efficiency and accuracy remain the subject of ongoing work. Third, exploring whether aspects of the multi-fidelity framework can be generalized to Navier–Stokes flows is an interesting but very challenging direction.

\section*{Acknowledgements}
This research was funded in part by the Wallenberg Academy Fellow grants KAW 2019.0190 and KAW 2024.0234. Part of the computations were enabled by resources provided by the National Academic Infrastructure for Supercomputing in Sweden (NAISS), funded by the Swedish Research Council.

We thank the reviewers for their valuable comments, which helped us improve this paper. We also thank Kilian Bideau for providing us with useful comments on a preliminary version of this manuscript.

\appendix
\renewcommand{\thesubsection}{\Alph{subsection}}

\section*{Appendix}
\subsection{Proof of \cref{prop:correctfull}}\label{app:proof}
The proof proceeds by induction on $j$, by showing that, for all $j$, the claim from \cref{prop:correctfull} holds, in addition to
\begin{equation}\label{eq:correctfullextra}
    \begin{cases}
        -\nabla\cdot\uT(\vu_{j+1/2},p_{j+1/2};\mu)=\vforce+2(\mu-\mu_F)\nabla\cdot\uE(\vv_{j+1/2}-\vphi_{j+1/2})&\text{in }\Omega_B\cup\Omega_F,\\
        \nabla\cdot\vu_{j+1/2}=0&\text{in }\Omega_B\cup\Omega_F,\\
        \jump{\vu_{j+1/2}}=\vzero&\text{on }\Gamma,\\
        \jump{\uT(\vu_{j+1/2},p_{j+1/2};\mu)\vn}=\left(\gamma\kappa\uI-2\jump{\mu}\uE(\vv_{j+1/2}-\vphi_{j+1/2})|_{\Omega_B}\right)\vn&\text{on }\Gamma,
    \end{cases}
\end{equation}
with $(\vu_{j+1/2},p_{j+1/2})$ from \cref{algo:full}.

\paragraph{Base case for $j=0$.} We initialize $\vphi_{-1/2}=\vv_{-1/2}=\vu_0$. The first claim, involving $(\vv^{(0)},q^{(0)})$, follows by subtracting \cref{eq:fem} and \cref{eq:fematgamma} from \cref{eq:cutfem} equation by equation, cf.~the inductive step.

In addition, we can show that \cref{eq:correctfullextra} holds for $j=0$ by linearity and by definition of $\vphi_{1/2}$. For instance, fix one of the interfaces $\Gamma_i$, $i=1,\ldots,M$. On $\Gamma_i$, by definition of $\vv_{0,i'}$ for $i'\neq i$, we have
\begin{align*}
    \jump{\uT(\vv_{0,i'},q_{0,i'};\mu)\vn}=&\jump{\uT(\vv_{0,i'},0;\mu-\mu_F)\vn}+\jump{\uT(\vv_{0,i'},q_{0,i'};\mu_F)\vn}\\
    =&2(\mu_F-\mu|_{\Omega_B^i})\uE(\vv_{0,i'})|_{\Omega_B^i}\vn+\vzero=2\jump{\mu}\uE(\vv_{0,i'})|_{\Omega_B^i}\vn,
\end{align*}
whereas, by definition of $\vv_{0,i}$, cf.~\cref{eq:singledropfull},
\begin{equation*}
    \jump{\uT(\vv_{0,i},q_{0,i};\mu)\vn}=-2\jump{\mu}\uE(\vphi_{-1/2})|_{\Omega_B^i}\vn.
\end{equation*}
From here we can use the previous claim and the definitions of $\vu$ and $\vphi_{1/2}$ to show that
\begin{align*}
    \jump{\uT(\vu_{1/2},p_{1/2};\mu)\vn}=&\jump{\uT(\vu_0,p_0;\mu)\vn}+\sum_{i'=1}^M\jump{\uT(\vv_{0,i'},q_{0,i'};\mu)\vn}\\
    =&\jump{\uT(\vu,p;\mu)\vn}-\jump{\uT(\vv^{(0)},q^{(0)};\mu)\vn}+\jump{\uT(\vv_{0,i},q_{0,i};\mu)\vn}+\sum_{i'\neq i}\jump{\uT(\vv_{0,i'},q_{0,i'};\mu)\vn}\\
    =&\gamma\kappa\vn+2\jump{\mu}\uE(\vphi_{-1/2})|_{\Omega_B^i}\vn-2\jump{\mu}\uE(\vphi_{-1/2})|_{\Omega_B^i}\vn+2\jump{\mu}\sum_{i'\neq i}\uE(\vv_{0,i'})|_{\Omega_B^i}\vn\\
    =&\gamma\kappa\vn+2\jump{\mu}\uE({\textstyle\sum_{i'\neq i}\vv_{0,i'}})|_{\Omega_B^i}\vn=\gamma\kappa\vn+2\jump{\mu}\uE(\vphi_{1/2}-\vv_{1/2})|_{\Omega_B^i}\vn,
\end{align*}
with the final step following by inspection of $\vphi_{1/2}|_{\Omega_B^i}$, cf.~\cref{rem:phi}. A similar derivation can be carried out also within the droplet $\Omega_B^i$.

\paragraph{Inductive step for $j>0$.} The claim in \cref{prop:correctfull}, involving $(\vv^{(j)},q^{(j)})$, can be proved by leveraging the linearity of the PDE and the inductive hypothesis on both the main claim of \cref{prop:correctfull} and \cref{eq:correctfullextra}. For instance, by construction of $\vv_{j-1/2}$, on $\Gamma$ we have
\begin{align*}
    \jump{\uT(\vv_{j-1/2},q_{j-1/2};\mu)\vn}=&\jump{\uT(\vv_{j-1/2},0;\mu-\mu_F)\vn}+\jump{\uT(\vv_{j-1/2},q_{j-1/2};\mu_F)\vn}\\
    =&2(\mu_F-\mu|_{\Omega_B})\uE(\vv_{j-1/2})|_{\Omega_B}\vn+\vzero=2\jump{\mu}\uE(\vv_{j-1/2})|_{\Omega_B}\vn,
\end{align*}
so that, by inductive assumption,
\begin{align*}
    \jump{\uT(\vu_j,p_j;\mu)\vn}=&\jump{\uT(\vu_{j-1/2},p_{j-1/2};\mu)\vn}+\jump{\uT(\vv_{j-1/2},q_{j-1/2};\mu)\vn}\\
    =&\gamma\kappa\vn+2\jump{\mu}\uE(\vphi_{j-1/2}-\vv_{j-1/2})|_{\Omega_B}\vn+2\jump{\mu}\uE(\vv_{j-1/2})|_{\Omega_B}\vn\\
    =&\gamma\kappa\vn+2\jump{\mu}\uE(\vphi_{j-1/2})|_{\Omega_B}\vn.
\end{align*}
Recalling the definition of $(\vu,p)$, we deduce that, on $\Gamma$,
\begin{equation*}
    \jump{\uT(\vv^{(j)},q^{(j)};\mu)\vn}=\jump{\uT(\vu,p;\mu)\vn}-\jump{\uT(\vu_j,p_j;\mu)\vn}=-2\jump{\mu}\uE(\vphi_{j-1/2})|_{\Omega_B}\vn.
\end{equation*}
A similar derivation can be carried out also within the droplet $\Omega_B^i$.

From here, by following the same steps as in the base case we can show that \cref{eq:correctfullextra} holds true. This is sufficient for the next induction step.

\subsection{Proof of \cref{lem:singlewellposed}}\label{app:prooftwo}
The correction fields $\vv_{j,i}$ solve \cref{eq:singledroplet} with $\vphi=\vphi_{j-1/2}$ and $\widetilde\gamma=0$. They may be expressed in terms of integral operators corresponding to the Newtonian and single-layer potentials \cite{Pozrikidis1992}:
\begin{equation}\label{eq:biedecay}
    \vv_{j,i}(\vx)=\left.\left[\mathcal N_{\Omega_B^i}\left(2(\mu-\mu_F)\nabla\cdot\uE(\vphi_{j-1/2})\right)\right]\right|_{\vx}+\left.\left[\mathcal S_{\partial\Omega_B^i}\left(-2\jump{\mu}\uE(\vphi_{j-1/2})|_{\Omega_B}\vn\right)\right]\right|_{\vx}\quad\forall\vx\in\mathbb R^d\setminus\overline{\Omega_B^i},
\end{equation}
where
\begin{equation*}
    \left.\left[\mathcal N_{\Omega_B^i}\left(\mathbf{f}\right)\right]\right|_{\vx}=\int_{\Omega_B^i}\textup{G}(\vx-\vy)\mathbf{f}(\vy)\textrm{d}\vy\quad\text{and}\quad\left.\left[\mathcal S_{\partial\Omega_B^i}\left(\mathbf{f}\right)\right]\right|_{\vx}=\int_{\partial\Omega_B^i}\textup{G}(\vx-\vy)\mathbf{f}(\vy)\textrm{d}S(\vy),
\end{equation*}
with the ``Stokeslet'' acting as singular kernel
\begin{equation*}
    \textup{G}(\vx)=\begin{cases}
        \frac1{4\pi}\left(-\log|\vx|\textup{I}+\frac1{|\vx|^2}\vx\otimes\vx\right)&\quad\text{if }d=2,\\
        \frac1{8\pi}\left(\frac1{|\vx|}\textup{I}+\frac1{|\vx|^3}\vx\otimes\vx\right)&\quad\text{if }d=3.
    \end{cases}
\end{equation*}
Under the gap condition \cref{eq:gap}, the kernels are evaluated at nonzero inputs $\vx-\vy$, which makes the integral operators non-singular and bounded in the usual Sobolev norms. Moreover, (distributional) derivatives of the integral operators are also well-defined and bounded, which allow bounding (distributional) derivatives of $\vv_{j,i}$ at locations away from $\Omega_B^i$.

This can be used in combination with standard trace inequalities to obtain bounds on the boundary-integral operator norms. For instance,
\begin{align*}
    \norm{\mathcal N_{\Omega_B^i}}_{H^{-1}(\Omega_B^i)\to H^{1/2}(\Gamma_D)\times H^{-1/2}(\Gamma_N)}:=&\sup_{\mathbf{f}\in H^{-1}(\Omega_B^i)\setminus\{0\}}\frac{\norm{\mathcal N_{\Omega_B^i}\left(\mathbf{f}\right)}_{H^{1/2}(\Gamma_D)\times H^{-1/2}(\Gamma_N)}}{\norm{\mathbf{f}}_{H^{-1}(\Omega_B^i)}}\\
    \leq&\sup_{\mathbf{f}\in H^{-1}(\Omega_B^i)\setminus\{0\}}\frac{\norm{\mathcal N_{\Omega_B^i}\left(\mathbf{f}\right)}_{H^1(U)}}{\norm{\mathbf{f}}_{H^{-1}(\Omega_B^i)}}\sup_{\vv\in H^1(U)\setminus\{0\}}\frac{\norm{\vv}_{H^{1/2}(\Gamma_D)\times H^{-1/2}(\Gamma_N)}}{\norm{\vv}_{H^1(U)}}\\
    \leq&\sup_{\mathbf{f}\in H^{-1}(\Omega_B^i)\setminus\{0\}}\frac{\norm{\mathcal N_{\Omega_B^i}\left(\mathbf{f}\right)}_{H^1(U)}}{\norm{\mathbf{f}}_{H^{-1}(\Omega_B^i)}}\sup_{\vv\in H^1(U)\setminus\{0\}}\frac{\norm{\vv}_{H^{1/2}(\Gamma_D)\times H^{-1/2}(\Gamma_N)}}{\norm{\vv}_{H^1(U)}}\\
    =&\norm{\mathcal N_{\Omega_B^i}}_{H^{-1}(\Omega_B^i)\to H^1(U)}\norm{\text{trace}}_{H^1(U)\to H^{1/2}(\Gamma_D)\times H^{-1/2}(\Gamma_N)},
\end{align*}
where $U$ is the intersection between $\Omega$ and a suitable neighborhood of $\partial\Omega$, such that there is a positive gap between $U$ and $\Omega_B^i$, e.g., $U$ may be a thin ``strip'' whose boundary includes $\partial\Omega$. The boundedness of the Newtonian potential's $H^{-1}(\Omega_B^i)\to H^1(U)$ operator norm is guaranteed by such gap \cite{Pozrikidis1992}.

This can be used to show \cref{eq:singlewellposedboundary}:
\begin{align*}
    R_{j+1/2}=&\norm{\sum_{i=1}^M\vv_{j,i}}_{H^{1/2}(\Gamma_D)}+\norm{\uT(\sum_{i=1}^M\vv_{j,i},\sum_{i=1}^Mq_{j,i};\mu)\vn}_{H^{-1/2}(\Gamma_N)}\\
    \leq&\sum_{i=1}^M\left(\norm{\vv_{j,i}}_{H^{1/2}(\Gamma_D)}+\norm{\uT(\vv_{j,i},q_{j,i};\mu)\vn}_{H^{-1/2}(\Gamma_N)}\right)\\
    \leq&\sum_{i=1}^M\max\left\{\norm{\mathcal N_{\Omega_B^i}}_{H^{-1}(\Omega_B^i)\to H^{1/2}(\Gamma_D)\times H^{-1/2}(\Gamma_N)},\norm{\mathcal S_{\partial\Omega_B^i}}_{H^{-1/2}(\partial\Omega_B^i)\to H^{1/2}(\Gamma_D)\times H^{-1/2}(\Gamma_N)}\right\}R_{j,i},
\end{align*}
cf.~\cref{eq:fixedfullloc}. Setting $C_{\Omega_B\to\partial\Omega}$ as the maximum over $i$ of the above operator norms yields the claim.

Similarly, given any $i\neq i'$, the positive gap between $\Omega_B^i$ and $\Omega_B^{i'}$ makes the operator norms involved in the following inequality bounded:
\begin{multline*}
    \norm{\vv_{j,i}}_{H^1(\Omega_B^{i'})}+\norm{\uT(\vv_{j,i},q_{j,i};\mu)|_{\Omega_B}\vn}_{H^{-1/2}(\partial\Omega_B^{i'})}\leq\\
    \leq\max\left\{\norm{\mathcal N_{\Omega_B^i}}_{H^{-1}(\Omega_B^i)\to H^1(\Omega_B^{i'})\times H^{-1/2}(\partial\Omega_B^{i'})},\norm{\mathcal S_{\partial\Omega_B^i}}_{H^{-1/2}(\partial\Omega_B^i)\to H^1(\Omega_B^{i'})\times H^{-1/2}(\partial\Omega_B^{i'})}\right\}R_{j,i},
\end{multline*}
which proves \cref{eq:singlewellposedvolume}.

It now remains to show that the boundary-integral operator norms become arbitrarily small for sufficiently large gaps. To show this, we can rely on standard estimates for the decay rates of the kernels $\textup{G}(\vx-\vy)$, namely, the Stokeslet, as $\|\vx-\vy\|\to\infty$, which have direct implications on the decay rate of the operator norms of $\mathcal N$ and $\mathcal S$ as the distance between the evaluation point $\vx$ and the integral variable $\vy$ increases. We provide below only the main details, referring to \cite{Pozrikidis1992,forceMoments} for a more thorough discussion on this.

In dimension $d=3$, the Stokeslet and both integral operators $\mathcal N_{A}$ and $\mathcal D_{\partial A}$ have a decay of order $\mathcal{O}(\delta^{-1})$, with $\delta$ being a lower-bound on $\text{dist}(A,\vx)$, with $\vx$ being the evaluation point, cf.~\cref{eq:biedecay}.

On the other hand, the Stokeslet magnitude does not decay with distance if $d=2$. However, we can still show that the magnitude of $\vv_{j,i}(\vx)$ decays to zero as $\|\vx\|\to\infty$ through a modified argument, which is standard in FMM theory. Specifically, in \cref{eq:biedecay}, a zero ``net force'' is acting on $\Omega_B^i$ \cite{forceMoments}, in the sense that the argument of the Newtonian potential operator is ``balanced'' by the argument of the single-layer potential through the divergence theorem:
\begin{equation*}
    \int_{\Omega_B^i}\left(2(\mu-\mu_F)\nabla\cdot\uE(\vphi_{j-1/2})\right)+\int_{\partial\Omega_B^i}\left(-2\jump{\mu}\uE(\vphi_{j-1/2})|_{\Omega_B}\vn\right)=0.
\end{equation*}
For this reason, the logarithmic and constant Stokeslet contributions disappears as $\|\vx\|\to\infty$, leaving only higher-order terms (the stresslet) which decay with order $\mathcal{O}(\delta^{-1})$.

In both cases, if $\delta$ is sufficiently large, one can make the integral operators' norms small enough to compensate the trace norms involved in \cref{eq:singlewellposedboundary,eq:singlewellposedvolume}, ultimately yielding the claim. The linear scaling of $\delta_\star$ with respect to droplet size is a straightforward consequence of affine changes of variable in the integral operators, which makes droplet size appear as the determinant of the transformation's Jacobian.

\subsection{Proof of \cref{thm:convergenceD}}\label{app:proofthree}
First, note that bounds \cref{eq:singlewellposedfull,eq:singlewellposedboundary,eq:singlewellposedvolume} remain valid even in the subspace-approximation case. Indeed, the same arguments used to derive them can be applied when the additional action of subspace-projector $\mathcal P_N$ is present, since, having unit operator norm, it does not affect the bounds through additional operator norms $\norm{\mathcal P_{N,i}}_{H^1(\Omega_B^i)\to H^1(\Omega_B^i)}$, $i=1,\ldots,M$.

Following this preliminary observation, we prove the claim by showing that $\vu_j$, or rather $\vv^{(j)}=\vu-\vu_j$, forms an $H^1(\Omega)$-Cauchy sequence. By the well-posedness of \cref{eq:cutfem}, it suffices to show that the residuals, i.e., the right-hand sides in \cref{eq:correctfullD}, form an $(H^{-1}(\Omega_B)\times H^{-1/2}(\Gamma))$-Cauchy sequence.

Using the notation $R^{(i)}$ from the proof of \cref{thm:convergence} to denote the local residual on $\overline{\Omega_B^i}$, the Cauchy increment of the residual between iterations $j_1$ and $j_2>j_1$ is
\begin{multline*}
    \sum_{i=1}^MR^{(i)}\left(2(\mu-\mu_F)\left(\vphi_{j_2-1/2}+(\mathcal I-\mathcal P_{N,i})\sum_{j'=j_1+1}^{j_2-1}\vphi_{j'-1/2}-\mathcal P_{N,i}\vphi_{j_1-1/2}\right)\right)\leq\\
    \leq C_\mu\sum_{i=1}^M\left(R^{(i)}\left(\vphi_{j_2-1/2}\right)+C_{\text{equiv}}\sum_{j'=j_1+1}^{j_2-1}R^{(i)}\left(\vphi_{j'-1/2}\right)+C_{\text{equiv}}R^{(i)}\left(\vphi_{j_1-1/2}\right)\right)\leq\\
    \leq C_\mu\max\{1,C_{\text{equiv}}\}\sum_{j'=j_1}^{j_2}\sum_{i=1}^MR^{(i)}\left(\vphi_{j'-1/2}\right),
\end{multline*}
thanks to the triangular inequality. Above, $C_\mu$ is as in the proof of \cref{thm:convergence}, whereas $C_{\text{equiv}}$ incorporates the norm-equivalence constants needed to switch from $R^{(i)}$ to the $H^1(\Omega_B^i)$-norm and back. The latter topology is needed in order to exploit the boundedness (with unit constant) of projections $(\mathcal I-\mathcal P_{N,i})$ and $\mathcal P_{N,i}$, e.g.,
\begin{equation*}
    R^{(i)}\left(\mathcal P_{N,i}\vphi_{j_1-1/2}\right)\leq C_{\text{equiv}}'\norm{\mathcal P_{N,i}\vphi_{j_1-1/2}}_{H^1(\Omega_B^i)}\leq C_{\text{equiv}}'\norm{\vphi_{j_1-1/2}}_{H^1(\Omega_B^i)}\leq C_{\text{equiv}}R^{(i)}\left(\vphi_{j_1-1/2}\right).
\end{equation*}

The claim then follows by showing that $\sum_iR^{(i)}\left(\vphi_{j-1/2}\right)$ decays quickly in $j$ for large enough gaps $\delta$. Recalling the expression for $\vphi_{j-1/2}$ in \cref{rem:phi}, this can be shown by following the same steps as in the proof of \cref{thm:convergence}. Indeed, $\vphi_{j-1/2}$ over $\Omega_B^i$ does not contain any self-interaction $i\to i$ terms, but only far interactions $i\neq i'\to i$ and boundary terms $\partial\Omega\to i$ based on far interactions $i\to\partial\Omega$, all of which can be made arbitrarily small through a large-gap assumption, cf.~\cref{lem:singlewellposed}.

\subsection{Change of coordinates in second-fidelity problem}\label{app:change}
For an arbitrary $\vx_i\in\bR^d$ and $r_i>0$, define the affine function $\eta_i(\vx)=r_i\vx+\vx_i$, which maps the unit disk $\Omega_B^\sharp$ to $\Omega_B^i$. We consider problem \cref{eq:singledroplet} with $\Omega_B^i=B(\vx_i;r_i)$, $\vphi\in[H^1(\Omega_B^i)]^d$ smooth enough, and $\mu=\mu_F+(\mu(\vx_i)-\mu_F)\chi_{\Omega_B^i}$.

Let $\widehat\vv:=\vv\circ\eta_i$, $\widehat q:=r_i q\circ\eta_i$, $\widehat\mu:=\mu\circ\eta_i$, and $\widehat\vphi:=\vphi\circ\eta_i$. Moreover, by the chain rule, let $\widehat\nabla:=r_i\nabla$ and $\widehat\uT(\widehat\vv(\cdot),\widehat q(\cdot);\widehat\mu)=r_i\uT(\vv(\cdot),q(\cdot);\mu)\circ\eta_i$, which involve differential operators in the reference coordinates $\eta_i^{-1}(\vx)$. It follows that
\begin{equation*}
    \begin{cases}
        -r_i^{-1}\widehat\nabla\cdot\left(r_i^{-1}\widehat\uT(\widehat\vv,\widehat q;\widehat\mu)\right)=(\widehat\mu-\mu_F)r_i^{-1}\widehat\nabla\cdot\left(r_i^{-1}\widehat\uE(\widehat\vphi)\right)&\textup{in }\bR^d\setminus\partial \Omega_B^\sharp,\\
        r_i^{-1}\widehat\nabla\cdot\widehat\vv=0&\textup{in }\bR^d\setminus\partial \Omega_B^\sharp,\\
        \jump{\widehat\vv}=\vzero&\textup{on }\partial \Omega_B^\sharp,\\
        \jump{r_i^{-1}\widehat\uT(\widehat\vv,\widehat q;\widehat\mu)\vn}=r_i^{-1}\left(\gamma\kappa\uI-\jump{\widehat\mu}\widehat\uE(\widehat\vphi)\right)\vn&\textup{on }\partial \Omega_B^\sharp,\\
        \lim_{\norm{\vx}\to\infty}\widehat\vv(\vx)=\vzero.
    \end{cases}
\end{equation*}
After canceling out the powers of $r_i$, we see that this problem is equivalent to \cref{eq:singledroplet} with $\Omega_B^i=\Omega_B^\sharp$, $\widehat\vphi$ as forcing flow field, and $\mu=\widehat\mu$.

\subsection{Polynomial basis in dimension two}\label{app:poly}
In \cref{algo:full}, we only apply $\text{S\_droplet}$, i.e., $\mathcal S$ from \cref{eq:calS}, to fields $\vphi$ that satisfy the homogeneous Stokes equation over $\Omega_B^i$, i.e., that belong to the space of smooth, \emph{compatible} flow fields
\begin{equation*}
    \mathcal V=\left\{\vv\in[C^3(\bR^d)]^2,\nabla\cdot\vv\equiv0,\nabla\times\nabla\cdot\uE(\vv)\equiv0\right\}.
\end{equation*}
The last condition above encodes a compatibility condition, ensuring that, for each element $\vv\in\mathcal V$, one can find a pressure field $q\in C^2(\bR^d)$ satisfying the homogeneous Stokes equation $\nabla\cdot\uT(\vv,q;\mu_F)=\vzero$ \cite{Muskhelishvili1953,happel_low_1983}. To compute a (polynomial) basis of $\mathcal V$ in $d=2$ dimensions, we use stream functions.

Recall that, in $\bR^2$, $\mathcal V=\nabla^\perp\mathcal W$, with $\nabla^\perp=[-\partial_{x_2},\partial_{x_1}]^\top$ and with $\mathcal W$ being the space of biharmonic scalar fields \cite{happel_low_1983}
\begin{equation*}
    \mathcal W=\left\{\phi\in C^4(\bR^2),\nabla\cdot\nabla(\nabla\cdot\nabla\phi)\equiv0\right\}.
\end{equation*}
A hierarchical polynomial basis of $\mathcal W$ can be easily constructed as \cite{Muskhelishvili1953}
\begin{equation*}
    \mathcal W=\text{span}\bigcup_{n=0}^\infty\left\{\rho\psi\in\mathbb P^n(\bR^2),\text{deg}(\rho\psi)=n,\rho(\vx)\in\{1,x_1^2+x_2^2\},\nabla\cdot\nabla\psi\equiv0\right\}.
\end{equation*}
In turn, harmonic polynomials $\psi$ on $\bR^2$ are easily expressible through holomorphy arguments as real or imaginary parts of polynomials in $(x_1+\textup{i}x_2)$. Accordingly, a hierarchical (in total degree) basis of $\mathcal W$ is found:
\begin{multline*}
    \mathcal W=\text{span}\bigg(\left\{1\right\}\cup\left\{x_1,x_2\right\}\cup\left\{x_1^2-x_2^2,2x_1x_2,x_1^2+x_2^2\right\}\cup\\\cup\bigcup_{n=3}^\infty\left\{\text{Re}\left((x_1+\textup{i}x_2)^n\right),\text{Im}\left((x_1+\textup{i}x_2)^n\right),(x_1^2+x_2^2)\text{Re}\left((x_1+\textup{i}x_2)^{n-2}\right),(x_1^2+x_2^2)\text{Im}\left((x_1+\textup{i}x_2)^{n-2}\right)\right\}\bigg).
\end{multline*}
From here, a polynomial basis for the space $\mathcal V=\nabla^\perp\mathcal W$ of compatible flow fields can be derived by linearity of the differential operator involved. 

Note that, since differentiations is involved to extract flow fields from stream functions, and since $\vphi$ enters \cref{eq:singledroplet} only through its strain field $\uE(\vphi)$, constant and linear basis functions of $\mathcal W$ are irrelevant and can be discarded. Additionally, the quadratic stream function $\phi(\vx)=x_1^2+x_2^2$ yields a strain-free rotation field and can thus be ignored:
\begin{equation*}
    \uE\left(\nabla^\perp(x_1^2+x_2^2)\right)=\uE\left([-2x_2,2x_1]^\top\right)=\begin{pmatrix}0&0\\0&0\end{pmatrix}.
\end{equation*}
Truncating the basis of $\mathcal W$ at degree $(D+1)$, with $D\geq1$, yields a degree-$D$ basis for $\mathcal V$ with dimension $4D-2$.

As an example, we report the hierarchical basis of $\mathcal V$ up to degree $D=3$ in \cref{tab:basis}. The exact enumeration and implementation of the basis elements $\widehat\vphi_k$ used in our numerical experiments is also available in the given code repository, cf.~\cref{sec:numerics}. See also \cref{fig:basis}, where some representative elements are visualized.

\begin{table}[ht]
    \centering
    \begin{tabular}{c|c|c}
        $D=1$ & $D=2$ & $D=3$ \\
        \hline
        $\widehat\vphi_1(\vx)=[x_2,x_1]^\top$ & $\widehat\vphi_3(\vx)=\frac12[2x_1x_2,x_1^2-x_2^2]^\top$ & $\widehat\vphi_7(\vx)=\frac16[3x_1^2x_2-x_2^3,x_1^3-3x_1x_2^2]^\top$ \\
        $\widehat\vphi_2(\vx)=[-x_1,x_2]^\top$ & $\widehat\vphi_4(\vx)=\frac12[-x_1^2+x_2^2,2x_1x_2]^\top$ & $\widehat\vphi_8(\vx)=\frac16[-x_1^3+3x_1x_2^2,3x_1^2x_2-x_2^3]^\top$ \\
        & $\widehat\vphi_5(\vx)=\frac12[-2x_1x_2,3x_1^2+x_2^2]^\top$ & $\widehat\vphi_9(\vx)=\frac23[x_2^3,x_1^3]^\top$ \\
        & $\widehat\vphi_6(\vx)=\frac12[-3x_1^2-x_2^2,2x_1x_2]^\top$ & $\widehat\vphi_{10}(\vx)=\frac13[-x_1^3-3x_1x_2^2,3x_1^2x_2+x_2^3]^\top$ \\
    \end{tabular}
    \caption{Polynomial basis of $\mathcal V$ up to degree $3$. As explained above, three strain-free fields (two with $D=0$, one with $D=1$) are not included since they belong to the kernel of $\mathcal S$, cf.~\cref{eq:calS}. In order to obtain cleaner, equivalent expressions, the basis elements have been rescaled by constants in some cases.}
    \label{tab:basis}
\end{table}

When precomputing the single-droplet solutions corresponding to the above basis functions, the computational effort may be reduced by exploiting the symmetry of stream functions. This can be used to cut the precomputation time in half.

\subsection{CutFEM setup}\label{app:cutfem}
Denote by $\mcK_h$ a regular partition of $\Omega$ which does not need to conform to the interfaces $\Gamma$ inside $\Omega$ but conforms to the polygonal boundary $\dw$. Associated with the subdomain $\Omega_B$ and $\Omega_F$, and with the interfaces $\Gamma$, we define the active meshes:
\begin{equation*}
	\mcK_{h,\bullet} = \left\{K \in \mcK_h,\ K \cap  \Omega_i \neq \emptyset\right\}, \ \bullet=B,F, \quad\text{and}\quad \mcK_{h,\Gamma} = \left\{K \in \mcK_h,\ |\Gamma \cap \bar{K}| > 0\right\}.
\end{equation*}
From the active meshes, we define the active domain 
\begin{equation*}
	\Omega_{h,\bullet} = \bigcup_{K \in \mcK_{h,\bullet}} K, \quad\bullet=B,F
\end{equation*}
and the approximate interface $\Gamma_h$ as a piecewise-polynomial approximation of $\Gamma$.

Let $\mcV_h^\vu$ and $\mcV_h^p$ be the Taylor-Hood elements P2-P1 on $\mcK_h$. The active finite-element space is then the restriction to the active domain:
\begin{equation*}
	\mcV_{h,\bullet}^\vu = \mcV_h^\vu|_{\Omega_{h,\bullet}}\quad\text{and}\quad\mcV_{h,\bullet}^p = \mcV_h^p|_{\Omega_{h,\bullet}} , \quad \bullet=B,F, 
\end{equation*}
and we let
\begin{align*}
V_h^\vu &= \mcV_{h,B}^\vu \times \mcV_{h,F}^\vu = \left\{ \vv_h = (\vv_{h,B}, \vv_{h,F}) , \ \vv_{h,\bullet}\in \mcV_{h,\bullet}^\vu, \ \bullet=B,F \right\},\\
V_h^p&   = \mcV_{h,B}^p \times \mcV_{h,F}^p = \left\{ q_h = (q_{h,B}, q_{h,F}) , \ q_{h,\bullet}\in \mcV_{h,\bullet}^p, \ \bullet=B,F \right\}.
\end{align*}

The weak formulation of \eqref{eq:cutfem} is: 
find $(\vu_h, p_h) \in V_h^{\vu} \times V_h^{p}$ such that
\begin{equation*}
a(\vu_h, \vv_h) + \tau_\vu s_\vu( \vu_h, \vv_h) + b(\vu_h, q_h) - b(\vv_h, p_h) + \tau_p s_p(p_h,q_h) = l(\vv_h, q_h)\quad\forall (\vv_h, q_h) \in V_h^{\vu} \times V_h^{p},
\end{equation*}
Above,
\begin{align*}
	a(\vu_h,\vv_h) & = \sum_{\bullet\in\{B,F\}}\prodscal{2\mu\uE(\vu_h)}{\uE(\vv_h)}_{\Omega_{h,\bullet}}
                     - \prodscal{\lbrace 2\mu \uE(\vu_h)  \vn \rbrace}{ \jump{\vv_h}}_{\Gamma_h}
                     + \prodscal{\jump{\vu_h}}{\lbrace 2\mu \uE(\vv_h)  \vn \rbrace}_{\Gamma_h}\\
                   & + \lambda_{\Gamma}\prodscal{\jump{\vu_h}}{\jump{\vv_h}}_{\Gamma_h}
                     - \prodscal{2\mu \uE(\vu_h)  \vn}{\vv_h}_{\Gamma_D}
                     + \prodscal{\vu_h}{2\mu \uE(\vv_h)  \vn}_{\Gamma_D}
                     + \lambda_D\prodscal{\vu_h}{\vv_h}_{\Gamma_D},\\
	b(\vu_h,q_h) & = \sum_{\bullet\in\{B,F\}}\prodscal{\nabla \cdot \vu_h}{q_h}_{\Omega_{h,\bullet}}
                 - \prodscal{\jump{\vu_h \cdot \vn}}{ \lbrace q_h\rbrace}_{\Gamma_h}
                 - \prodscal{\vu_h \cdot \vn}{q_h}_{\Gamma_D},\\
	l(\vv_h,q_h) & = \sum_{\bullet\in\{B,F\}}\prodscal{\vforce}{\vv_h}_{\Omega_h,\bullet}
                 + \prodscal{\gamma \kappa \vn}{\lbrace\vv_h\rbrace}_{\Gamma_h}
                 - \prodscal{\vf}{\vv_h}_{\Gamma_N}\\
                 &+ \prodscal{\vg }{ 2\mu \uE(\vv_h) \vn}_{\Gamma_D}
                 + \lambda_D\prodscal{\vg }{ \vv_h}_{\Gamma_D}
                 - \prodscal{\vg \cdot \vn }{q_h}_{\Gamma_D},
\end{align*}
with $(\cdot,\cdot)_A$ the $L^2(A)$-inner product and $\lbrace f\rbrace=\frac12(f|_{\Omega_B}+f|_{\Omega_F})$ the average across $\Gamma_h$.

Given the set of edges $\mathcal E_{h,\bullet}$ of each active mesh $\mcK_{h,\bullet}$, we define the stabilization terms $s_\vu$ and $s_p$ as 
\begin{align*}
s_\vu(\vu_h,\vv_h) & = \sum_{\bullet\in\{B,F\}} \sum_{e \in \mathcal E_{h,\bullet}}\left(h\prodscal{\jump{\partial_{\vn_e}\vu_{h,\bullet}}_e}{\jump{\partial_{\vn_e}\vv_{h,\bullet}}_e}_e + h^3\prodscal{\jump{\partial_{\vn_e}^2\vu_{h,\bullet}}_e}{\jump{\partial_{\vn_e}^2\vv_{h,\bullet}}_e}_e\right),\\
s_p(p_h,q_h) & = \sum_{\bullet\in\{B,F\}} \sum_{e \in \mathcal E_{h,\bullet}}h^3\prodscal{\jump{\partial_{\vn_e}p_{h,\bullet}}_e}{\jump{\partial_{\vn_e}q_{h,\bullet}}_e}_e,
\end{align*}
with $\vn_e$ the unit normal to edge $e$ and $\jump{\cdot}_e$ the jump across edge $e$, both with arbitrary but fixed sign.

\bibliographystyle{plain}
\bibliography{bibliography}

\end{document}